\documentclass[12pt,reqno]{article}
\usepackage{amssymb,amsmath,amsthm}
\usepackage[centering,textwidth=6.5in,textheight=9in]{geometry}
\usepackage{mathptmx}
\usepackage[normalem]{ulem}

\usepackage{hyperref}

\newtheorem{thm}{Theorem}[section]
\newtheorem{theorem}{Theorem}

\newtheorem{lemma}[thm]{Lemma}
\newtheorem{proposition}[thm]{Proposition}

\newtheoremstyle{mydefinition}{}{}{\normalfont}{0pt}{\scshape}{.}{.5em}{}
\theoremstyle{mydefinition}
\newtheorem{defn}{Definition}
\newtheorem{definition}[defn]{Definition}

\newtheoremstyle{myremark}{}{}{\small\normalfont}{0pt}{\small\scshape}{.}{.5em}{}
\theoremstyle{myremark}
\newtheorem{remark}{Remark}

\numberwithin{equation}{section}

\def\beq{\begin{equation}}
\def\eeq{\end{equation}}

\def\diam{{\rm diam}}

\def\supp{{\rm supp}}

\def\Cov{{\rm Cov}}

\def\bJ{\mathbf{J}}

\def\bN{\mathbf{N}}

\def\bc{\mathbf{c}}

\def\bh{\mathbf{h}}

\def\bk{\mathbf{k}}

\def\bp{\mathbf{p}}

\def\br{\mathbf{r}}
\def\bs{\mathbf{s}}

\def\bgamma{{\boldsymbol{\gamma}}}
\def\btheta{{\boldsymbol{\theta}}}

\def\b1{{\boldsymbol{1}}}

\def\cA{\mathcal{A}}

\def\cE{\mathcal{E}}
\def\cF{\mathcal{F}}
\def\cG{\mathcal{G}}
\def\cH{\mathcal{H}}

\def\cM{\mathcal{M}}

\def\cO{\mathcal{O}}

\def\cR{\mathcal{R}}

\def\cT{\mathcal{T}}
\def\cU{\mathcal{U}}

\def\cW{\mathcal{W}}

\def\cZ{\mathcal{Z}}

\def\tc{\tilde{c}}

\def\eps{\varepsilon}

\def\eps{\varepsilon}

\def\pa{\partial}

\def\ds{\displaystyle}

\let\emptyset\varnothing

\begin{document}

\title{Decay of Correlations for Unbounded Observables}

\author{Fang Wang \thanks{School of Mathematical Sciences, Capital Normal University, Beijing, 100048, P.R. China.  Email: fangwang@cnu.edu.cn.}
\and Hong-Kun Zhang\thanks{Department of Mathematics and Statistics, University of Massachusetts,  Amherst MA 01003. Email: hongkun@math.umass.edu. }
\and Pengfei Zhang \thanks{Department of Mathematics, University of Oklahoma, Norman, OK 73019. Email: pengfei.zhang@ou.edu. }}

\date{}

\maketitle

\begin{abstract}
In this article, we study the decay rates of the correlation functions for
a hyperbolic system  $T: M \to M$ with singularities that
preserves a unique mixing SRB measure $\mu$.
We prove that, under some general assumptions, the correlations
$ \Cov(f,g)$ decay exponentially as $n\to \infty$ for each
pair of piecewise H\"older observables $f, g\in L^p(\mu)$ and for each $p>1$.
As an application, we prove that
the autocorrelations of the first return time functions decay exponentially
for the induced maps of various billiard systems, which include the semi-dispersing billiards on a rectangle,
billiards with cusps, and Bunimovich stadia (for the truncated first return time functions).
These estimates of the decay rates of autocorrelations of the first return time functions for the induced maps have an essential importance in the study of the statistical properties of nonuniformly hyperbolic systems (with singularities).
\vskip.5cm

\noindent \textbf{Keywords:} Decay of correlations, unbounded observables, infinite variance,
hyperbolic systems, singularity, coupling method, first return time.
\end{abstract}

\section{Introduction}

Let $M$ be a smooth, compact Riemannian surface (possibly with boundary),
$ S_{0} \subset M$ be a compact subset consisting of at most countably many connected smooth curves,
$T$ be a smooth diffeomorphism from  $M\backslash  S_{0}$ onto its image.
Let $\mu$ be a  $T$-invariant probability measure on  $M$. Given a pair of integrable
functions $f,g \in L^1(\mu)$, and $n\in \mathbb{N}$, the {\em correlation
function} $ \Cov_n(f,g)$ of $f$ and $g$ at time $n$ is defined as
\[ \Cov_n(f,g) = \int_{M} f\circ T^n\,\cdot g\, d\mu -\int_{M} f\, d\mu\cdot
\int_{M} g\, d\mu,\] provided $\int_{M} f\circ T^n\,\cdot g\, d\mu<\infty$.
Then the dynamical system $(T,\mu)$ is said to be {\it mixing} if the correlations
 \beq\label{Cto0}
  \Cov_n(f,g)\to 0,\  \text{ as }\ n\to\infty,
 \eeq
for all pairs of $f,g \in L^2(\mu)$. The mixing rate of the system
$(T,\mu)$ is characterized by the \emph{decay rate of correlations}, i.e.,
by the speed of convergence of (\ref{Cto0}) for ``regular enough'' and bounded functions.
\vspace{.2cm}

If $\pa M=\emptyset$ and $T$ is a  diffeomorphism on $M$,
then the term ``regular enough''
usually means  H\"older continuity, see \cite{B1}.
For systems with
singularities, even if we start with a smooth observable $f$ on $M$,
after $n$ iterations, for some $n\geq 1$,
$f\circ T^n$ may be smooth  only on each connected component of  
$M\setminus (S_0 \cup \cdots \cup T^{1-n}S_0)$. 
This gives us
a hint that a natural observable that is suitable for the study of
the decay rates of correlations for systems with singularities needs only to be
H\"{o}lder continuous on regions such that $T^n$ is smooth, for some $n\geq 1$. Recent work  on characterizing decay rates of correlations for pairs of piecewise
H\"{o}lder observables for hyperbolic systems can be found in
\cite{C99,C06,CD,D01,Y98,Y99,CZ09,DZ11,DZ13}, and the references therein. However, all these results only apply to bounded observables, and there are no general results about the rates of decay of correlations for $L^p(\mu)$ functions, with $p\in (1,\infty)$. \vspace{.2cm}

Many important observables for the dynamical system, being piecewise H\"{o}lder
continuous on $M$,  are indeed {\it unbounded},
and the decay rates of correlations for the unbounded observables
play a key role in the study of statistical properties of dynamical systems.  One important example is the free path function for the Sinai billiards with infinite horizon \cite{CM}, which is not even  an $L^2(\mu)$  function.  Another example is the index function
 \beq\label{returntime}
R(x):=\sum_{n=1}^{\infty}n\cdot \chi_{{D_n}}(x),
 \eeq
where $\{D_n, n\geq 1\}$ is the collection of components in $M\setminus (S_0\cup T^{-1}S_0)$.
This function  is unbounded whenever $M\setminus (S_0\cup T^{-1}S_0)$ has infinitely many of connected components.
Moreover,  $\ds R \notin L^2(\mu)$ if there exists $c>0$ such that $\mu(D_n)\geq c n^{-3}$ for all $n$ large enough.

More generally, in the study of  nonuniformly hyperbolic system with  singularities
$(\cM,\cF,\hat\mu)$, we usually apply the following {\it inducing scheme} to
obtain a {\it uniformly hyperbolic} map (with singularities) $(M, T, \mu)$, see \cite{Ma04,CZ05a}. One first locates a nice subset
$M\subset \cM$,
then defines the first return time function $R: M\to\mathbb{N}$  of the original system
  $\cF$ on $M$.   Let $\mu$ be the conditional measure on $M$ of the measure $\hat\mu$ on $\cM$, which is  preserved
by  the first return map $T(x):=\cF^{R(x)}(x)$.  In studying various probability limiting theorems such as the {Central Limit Theorem},
{Law of Large Deviations} and the {Almost Sure Invariance Principle},
etc., for $(\cM,\cF,\hat\mu)$ , it is usually sufficient to investigate the unbounded process $R\circ T^n$. Although there are some available estimations based on individual systems,  cf. \cite{BCD}, the question of
estimating decay rates of correlations for the first return time function is still  open for most
nonuniformly hyperbolic systems. \vspace{.2cm}

 Our goal in this work is to investigate the decay rates of correlations for
the unbounded observables on general systems with singularities. We prove the exponential decay
of correlations for {\it a large class} of {\it unbounded}, piecewise H\"{o}lder
continuous observables (under fairly general assumptions, see Section \ref{secSR}).
Indeed our
result is \textit{new} even for many smooth hyperbolic systems (by adding finitely many, connected, smooth fake singular curves).
Moreover, the investigations of the limiting theorems using
the results proved in this paper is currently under way.
\vspace{.2cm}

In proving the exponential decay of correlations, the existence of singularities
aggravates the analysis, and makes the standard approaches hard to apply. In
\cite{Y98}, Young developed a novel method, now known as \emph{Young's tower
construction}. She successfully applied this construction to Sina\v{\i}
billiards with finite horizon,  and proved the exponential mixing property of
these billiards. This method was later extended by Chernov (cf. \cite{C99})
to Sina\v{\i} billiards with infinite horizon. Another approach, the so-called
\emph{coupling method}, was also introduced by Young in \cite{Y99},
and later improved by Chernov and Dolgopyat in \cite{CD, D01} through the
tools of standard pairs. The coupling method is designed to geometrically
control the dependence between the past and the future. Recently, Demers
and Zhang were able to prove the  exponential decay of correlations for
rather general hyperbolic systems with singularities using the {\it spectral
gap method} for the dynamical transfer operators (cf. \cite{DZ11,DZ13,DZ14}).
All these methods
have been verified to be very efficient, and have led to many deep results.
\vspace{.2cm}

In this paper, we continue the investigation of the system $(M,T,\mu)$
studied in \cite{CZ09} under a similar setting.
We obtain the exponential decay  of correlations
for unbounded observables using the coupling scheme. One of the main tools in our proof is the innovated
Chernov's $\cZ$-function introduced   in \cite{CM, C99,C06} (see
Eq. (\ref{cZ}) for definition). We use a dedicated analysis on
properties of this $\cZ$-function, which is closely related to the {\it Growth
Lemma}. This enables us to overcome the difficulty in obtaining the decay
rates on general piecewise H\"older continuous observables in $L^p(\mu)$,
for $p>1$. We also construct a hyperbolic set $\cR^*$ with product
structure, such that the dynamical system $(M,T)$ can be identified as a
suspension based on $\cR^*$. The coupling method is our main approach, which has been extensively
used in the recent studies of systems with singularities  (cf. \cite{CD,CM,
D01, D04b, CZ09, Y99}).
\vspace{.2cm}

One of our main achievements in this paper is the discovery of
the condition (\ref{bound})  for the exponential decay
of   unbounded functions. We also provide \textbf{(H6)} in \S \ref{assumptions} , which is a sufficient condition for (\ref{bound}) applying to billiard systems. 
It provides a simple criterion and is easy to be verified for a broad class of nonuniformly hyperbolic systems. Future researchers in this area can start their work on the limiting theory and other statistical properties from this criterion.  As straightforward examples, we show that $\Cov_n(R,R)$ decay exponentially for several classical billiard systems, such as billiards with
cusps, semi-dispersing billiards on a rectangle, billiards with
flat point,  Bunimovich flower billiards, Bunimovich stadia and the Skewed stadia. 
\vspace{.2cm}

This paper is organized in the following way. In \S\ref{secSR}, we list the
standard hypotheses \textbf{(H1)--(H5)} {and a new hypothesis \textbf{(H6)}  }
for the system $(M,T,\mu)$, and then
state our main results on the decay rates of correlations of unbounded
observables. In \S \ref{sec:4}, we introduce the concepts of standard pairs
and standard families. The growth lemma is also introduced in this section,
which plays an important role when we analyze the decay rates of the
correlations for unbounded observables. In \S \ref{magnet} we construct a
hyperbolic magnet, which serves as the base of the coupling scheme.
The coupling lemma is also introduced in this section.
In \S 5, we first establish an
equidistribution result for unbounded observables, and then give the proof of
exponential decay property for unbounded observables. In \S \ref{infinitevariance}, the decay rates of autocorrelations for the first
return time functions are studied. We establish the criterion and discuss several billiard models as examples, such as
semi-dispersing billiards, Bunimovich billiards, and billiards with flat
points, which all have special
importance in the modern theory of billiards. \\

\noindent NOTATION: Throughout the paper we will use the following
convention: positive global constants whose exact values are unimportant,
will be denoted by $c$, $c_1, c_2$, $\cdots$ or $C$, $C_1$, $C_2$, $\cdots$.	
These letters may be assigned to different values in different equations
throughout the paper. On the other hand, we use the symbols: $\bs_0$, $\bs_1$,   $b$, $c_{A}$, $C_{A}$, $C_{r,\bp}$, $\btheta$, $\bgamma_0$ and $\bgamma_1$, $\cdots$  to denote those constants whose values are fixed in the whole paper. Most of these constants are defined in \S\ \ref{assumptions}, and $C_{\br}$ are defined in \S\ \ref{stan_pair}.

\section{Main results}
\label{secSR}

\subsection{List of assumptions}\label{assumptions}

Let $M$ be a smooth, compact Riemannian surface (possibly with boundary),
$d(\cdot\,,\cdot)$ be the geodesic distance between two points on $M$.
Given a compact smooth curve $W\subset M$,
we let $|W|$ be the length of $W$, and $m_W$ be the conditional Lebesgue measure on $W$.
Let $ S_{0} \subset M$ be a compact subset consisting of at most countably many connected smooth curves,
$T$ be a smooth diffeomorphism from  $M\backslash  S_{0}$ onto its image.
Let $ S_{1}= S_0\cup T^{-1} S_0$, which is called the {\it singularity set} of $T$.
In the following we list and briefly explain the assumptions for our main results.

\bigskip

\noindent \textbf{(H1)} Hyperbolicity of $T$.
There exist two families of continuous cones $C^u_x$ (unstable) and $C^s_x$
(stable)  in the tangent spaces $\cT_x M$ for all $x\in M\setminus S_1$ with the following
properties:
\begin{itemize}
\item[(i)] The angle between $C^u_x$ and $C^s_x$ is uniformly bounded away
from zero;
\vspace{.1cm}
\item[(ii)] $D_x T (C^u_x)\subset C^u_{ T x}$ and $D_x T (C^s_x)\supset
    C^s_{ T x}$  (whenever $D_xT$ exists);
\vspace{.1cm}
\item[(iii)] There exists a constant $\Lambda>1$, such that
 \[
\|D_x T(v)\|\geq \Lambda \|v\|,\ \text{ for all }  v\in C_x^u;\quad
\|D_xT^{-1}(v)\|\geq \Lambda \|v\|,\  \text{ for all }\ v\in C_x^s.
\]
\end{itemize}
\vspace{-.2cm}
\begin{definition}
A smooth curve $W\subset M$ is said to be an \emph{unstable curve},
if the tangent line $T_x W$ belongs to the unstable cone
$C^u_x$ at every point $x \in W$. An unstable curve $W$
is called an \emph{unstable manifold}
 If $T^{-n}(W)$ is an unstable curve for each $n \geq 0$. We use $W^u$ to denote an unstable manifold, and use $\mathcal{W}^u$ to denote the collection of all unstable manifolds.
\end{definition}
For convenience, we denote $ S_{-1}= S_0\cup T S_0$.
Stable curves and stable manifolds are defined similarly by considering $T^{-1}$.

\bigskip

{
\noindent  \textbf{(H2)}  Singularities and smoothness.
Curves in $ S_{\pm 1}$ terminate on each other, i.e.
the endpoints of each curve $W\in  S_{\pm 1}$ must lie on other curves in $ S_{\pm 1}$. Moreover,
\begin{enumerate}
\item[(i)]   curves in $ S_0$ are uniformly transverse to both stable and
    unstable cones, and smooth curves in $T^{-1} S_0$ ($T S_0$, respectively)
    are stable (unstable, respectively) curves.

\item[(ii)] the numbers of smooth components of $\partial D$,
where $D$ runs over the connected components
of $M\backslash S_1$, are     uniformly bounded.

\item[(iii)]  there exists  $\bgamma_1\in (0,1]$
    such that  the map $T$ is a $C^{1+\bgamma_1}$ diffeomorphism on each
component $D$ and  can be extended continuously\footnote{Note that
the extensions of $T$ to the closures of two neighboring components $D$
and $D'$ could be different.}
to the closure $\overline{D}$.

\item[(iv)] there exist constants $\bs_1\in (0,1)$ and $C_1>0$, such that
    for any $x\in M\setminus  S_{\pm 1}$,
     \beq\label{upper} \|D_x T \|\leq
    C_1\cdot d(x,  S_{\pm 1})^{-\bs_1}. \eeq
\end{enumerate}

Note that for Sinai billiards, the last condition holds for $\bs_1=\frac{1}{2}$,
see \cite{CM07, KS86}.
}

For any $n\geq 1$, let $\ds S_{n}=\cup_{m=0}^{n-1} T^{-m} S_{1}$
be the singularity set of $T^n$,
and $\ds S_{\infty}=\cup_{m\geq 0}  S_{m}$.  Similarly, we define
$ S_{-n}$ and $ S_{-\infty}$.
Note that for every
stable/unstable manifold $W$, the end points of $W$ are on the singular
curves $ S_{\pm \infty}$. Thus the assumption \textbf{(H2)}  implies that the angles
between both stable and
unstable manifolds with the singular curves in
$ S_{\pm 1}$ is greater than some constant $\alpha_0>0$ at their intersection points.\vspace{.2cm}

\begin{definition}
Given two points $x, y \in M$. Let $\bs_+(x,y)$ be the smallest integer
$n\geq 0$ such that $x$ and $y$ belong to distinct elements of
$M\setminus S_n$, which is called  the forward
\emph{separation time} of $x, y$.
The backward separation time $\bs_-(x,y)$ is defined by reversing the direction.
\end{definition}
\vskip.2cm

\noindent  \textbf{(H3)}  Regularity of  stable/unstable curves.
Any unstable (and stable) curve $W$ is regular in the following sense:
  \begin{itemize}
\item[(i)] \textit{Bounded curvature.} There exists a constant
$\bk_0>0$, such that  the curvature of $W$ is bounded
from above by $\bk_0$.

 \item[(ii)] \textit{Bounded distortion.} There exist two constants
     $C_{\bJ}>1$ and $\bgamma_0\in(0,1]$, such that for each unstable
     curve $W\subset M$, and each pair of points $x,\ y\in W$,
 \beq
	  \left|\,\ln J_W T^{-1}(x)-\ln J_W T^{-1}(y)\right|
		\leq	C_{\bJ}\, d(x,y)^{\bgamma_0}\label{distor0},
 \eeq
 where $J_W T^{-1}(x)$ is  the Jacobian
of $T^{-1}$ at $x\in W$ with respect to the Lebesgue measure on the unstable manifold. (the subscript $\bJ$ in the constant $C_{\bJ}$ stands for Jacobian.)

\item[(iii)]\textit{Absolute continuity.} For each pair of regular
unstable curves $W^1$ and $ W^2$, which are close enough, we define
$$
   W^i_{\ast} := \{ x\in W^i \colon
   W^s( x) \cap W^{3-i} \neq \emptyset\},
$$
for $i=1,2$. The stable holonomy map $\bh:W^1_{\ast}\to W^2_{\ast}$
along stable manifolds is absolutely continuous with uniformly bounded
Jacobian $J_{W^1_*}\bh$. Furthermore, there exists  $ \btheta\in(0,1)$
such that for each $x,\ y\in W^1_*$,
 \beq\label{cjchb}
 |\ln J_{W^1_*}\bh(y)-\ln J_{W^1_*}\bh(x)|\leq
 C_{\bJ}\cdot\btheta^{\bs_+(x,y)}.
 \eeq
\end{itemize}

\vskip.4cm

A $T$-invariant probability Borel measure $\mu$ is said to be an {\it SRB} measure  (short for
Sina\v{\i}-Ruelle-Bowen), if its basin  is  of positive Lebesgue
measure, and the
conditional measures of $\mu$ on unstable manifolds are absolutely
continuous with respect to the leaf volume (see for example \cite{Y02,D04b}).
It provides a way to visualize the equilibrium state for
physicists, by averaging the atomic measures along typical trajectories. The
{\it existence} and \emph{finiteness} of mixing SRB measures for hyperbolic systems with
singularities have been studied by Chernov, Dolgopyat, Pesin,
Sina\v{\i}, Young and many others, see \cite{ABV,CD, P92, Y98,DZ11, DZ13}.\\

\noindent \textbf{(H4)}  SRB measure.
The map $T$  preserves a unique, mixing SRB measure $\mu$.
Moreover, there exists $\bs_0\in (0,1]$, such that
 \beq\label{alpha}
  \mu(\{x\in M\,:\, d(x,  S_{\pm 1})\leq
    \eps\})\leq C_0\cdot \eps^{\bs_0}
  \eeq
for every $\eps>0$.

Since $ S_1$ may consists of countably many singular curves,
some of the stable/unstable curves may be
quite short. The following condition ensures that, on average, forward iterations of these unstable curves
are generally long enough.
See Lemma \ref{global} and its following remark for a
quantitative estimate.
\vspace{.2cm}

\noindent  \textbf{(H5)} One-step expansion.
Let $\bs_0\in (0,1]$ be given as in (\ref{alpha}). Then
\begin{align}\label{onestep}
\liminf_{\delta\to 0}\ \sup_{W\colon
  |W|<\delta}\sum_{V_{\alpha}}
 \left(\frac{|W|}{|V_{\alpha}|}\right)^{{\bs_0}}\cdot
 \frac{|T^{-1}V_{\alpha}|}{|W|}<1,
 \end{align}
where the supremum is taken over all unstable curves $W$ of length $|W|< \delta$,
and the summation is over the set of connected components of $TW\backslash S_{-1}$.

\vspace{.2cm}

\begin{remark}
Note that we may assume that the lengths of
unstable/stable curves $W\subset M$ are uniformly bounded by a small constant
$\bc_M\in(0,1)$. That is, $|W|<\bc_M$ for any unstable/stable manifold $W$. One can
always guarantee this by adding a finite number of grid lines to $ S_{\pm 1}$ satisfying
\textbf{(H2)}.
\end{remark}

\vskip.1cm

\noindent  \textbf{(H6)}
There exists an enumeration  $\{D_n: n\ge 1\}$  of the set of connected components of
$M\backslash S_{1}$ such that the following holds:
\begin{enumerate}
\item[(i)] there exist $\bs>0$ and $C_0>0$ such that
  \beq\label{alpha1}
\mu(D_n)\leq C_0 n^{-2-\bs},\, \, \text{ for every }\, n\geq 1.
  \eeq

\item[(ii)] There exist two constants $C_A>0$ and
a  measurable partition of  $D_n$ into unstable curves $\{W_{\alpha}: \alpha\in \cA_n\}$
for each $n\ge 1$  that
satisfies  
\begin{align}\label{part}
|W_{\alpha}|\leq C_A n^{-\bs-b},\,\,\,|TW_{\alpha}|\leq C_A n^{-d}\,\,\text{ and }\,\,\,\,\,\, 
 \mu(D_n)/|TW_{\alpha}|\leq C_A n^{-\bs-b},\,\,\,\,\, \forall \alpha\in \cA_n,
\end{align}
for some positive constants  $d\in[0,1]$ and $b\geq 1$, such that 
\beq\label{assumpb} 
  b+\bs  >  \tfrac{1}{\bs_0} ;\,\,\,\,\,\text{ and }     \bs\geq \bs_0(2-b) \text{ if $b\in [1,2)$}.
\eeq
\end{enumerate}

A natural enumeration of $M\backslash S_1$ exists if $T$ is the first return map of some nonuniformly
hyperbolic system (with singularity), see \S \ref{sub.app}.
Moreover, $T^{-1}$ is well defined on $T D_n$, and can be extended by
continuity to the closure $\overline{TD_n}$ for each $n\ge 1$.
Then $\{TW_{\alpha}, \alpha\in \cA_n\}$ is a partition of $TD_n$, $n\ge 1$.

Note that {\textbf{(H6)}} is a new assumption, which is not needed in our general Theorem 1. Indeed the condition (\ref{part}) is only used in prove (\ref{cZgbarg}); and the condition (\ref{assumpb}) is only  needed in  Proposition \ref{coro-main3}. These are  used in the proof of Theorem \ref{main1}.   
We will verify the condition   {\textbf{(H6)}} for all billiards considered in this paper
 in \S \ref{pro.typB}.

\subsection{Main results}

Now we consider any $\gamma\in[\bgamma_0,1]$, where $\bgamma_0$ is the constant in (\ref{distor0}) in the assumption of bounded distortion.
For any $p\in(1,\infty]$,  let $\cH_p(\gamma)$ be
the set of all real-valued functions $f : M\to\mathbb{R}$, which is
H\"older continuous on each component of $D_n\subset M\setminus S_1$, $n\ge 1$, with
\begin{align}
\|f\|_{C^{\gamma}}:= &  \sup_{n\geq 1}  \: \sup_{ x, y\in
D_n}\frac{|f(x)-f(y)|}{d(x,y)^{\gamma}}<\infty; \\
\|f\|_{p,u}:= & \sup_{\alpha\in \cA^u} \: \sup_{x\in W_{\alpha}} |f(x)| |W_{\alpha}|^{\bs_0/p}< \infty.\label{fdefp}
\end{align}
One can check that
the $\|f\|_{p,u}<\infty$   if and only if 
$\ds \|f\|_p := \Big(\int |f(x)|^p d\mu(x) \Big)^{1/p}<\infty$, by using the Holder continuity of $f$ on each unstable curve $W_{\alpha}$ and the regularity of the density of conditional measures of the SRB measure $\mu$. 
We choose $\|f\|_{p,u}$ over $\|f\|_{p}$ since it is slightly more convenient for our discussion.
Moreover, it is easy to see that it induces a norm for functions $f\in \cH_p(\gamma)$ ,where
$$\|f\|_{p,\gamma}:=\|f\|_{p,u}+\|f\|_{C^{\gamma}}.$$

We give a sufficient condition under which the
correlation decays exponentially for $f,g\in \cH_{p}(\gamma)$.
Let $C_{\bp}$ and $C_{\bc}$ be two constants given 
by \eqref{Cp1} and  \eqref{ctail1}, respectively, and  
$C(\gamma, s)=\frac{1}{1-\Lambda^{-\gamma}} + \frac{1}{1-\Lambda^{-s}}$,
where $s$ is given by Lemma \ref{coupling1}.

\begin{theorem}\label{main2}
Assume \textbf{(H1)-(H5)} hold. Let $p\in (1,\infty]$,   and
 $f, g\in \cH_{p}(\gamma)$. Suppose that there exists $\eps>0$ such that
\begin{equation}\label{bound}
C_{g,\eps}(f) := \sup_{n\ge 1} \int_M |f\circ T^{n}|^{1+\eps}\cdot |g| \,d\mu< \infty.
\end{equation}
Then there exists a  constant $\vartheta\in(0,1)$  depending on $\eps, \gamma, p$ such that
 \begin{align}
 | \Cov_n(f,g)| \leq   
 [ 2C_{g,\eps}(f) + 4(2C(\gamma, s)C_{\bc}+c_M+C_{\bp}) \|f\|_{p,\gamma}\|g\|_{p.\gamma} ]\vartheta^n,
 \end{align}
 for any $n\ge 1$.
\end{theorem}

The assumption \eqref{bound} is rather crucial here.
The key point is that $g$ is in the same space as $f\in\cH_p(\gamma)$, instead of the dual space $\cH_{q}(\gamma)$, where $\frac{1}{p} + \frac{1}{q}=1$.
This is important for us to tackle the case $g=f$ in
the study of the decay rates of autocorrelations.\\

Next we define a smaller class of observables. 
For each $p>1$, any  $\kappa\in [0, p]$,
let $\cH_{\kappa,p}(\gamma)$ be the collection of piecewise H\"older
functions $f\in \cH_p(\gamma)$,
such that
\beq\label{Hkappap}
K_f:=\sup_{n\ge 1} \sup_{x\in D_n}|f(x)|\cdot n^{-\kappa/p} < \infty.
\eeq
If $p=\infty$, set $K_f=\|f\|_{\infty}$ for any  $f\in \cH_{\infty}(\gamma)$.

One can check easily that under assumption (\textbf{H6}), $K_f<\infty$ implies that 
\begin{align*}
\|f\|_{p,u} &= \sup_{\alpha\in \cA^u} \sup_{x\in W_{\alpha}} |f(x)||W_{\alpha}|^{\bs_0/p}\
\leq\sup_{n\geq 1} \sup_{x\in TD_n} |f(x)| |TW_{\alpha}|^{\bs_0/p} \\
& \leq C_A^{\bs_0/p}\sup_{n\geq 1} \sup_{x\in TD_n} |f(x)|n^{-d\cdot \bs_0/p}
=C_A^{\bs_0/p}\ K_f,
\end{align*}
where we used the fact that $d\bs_0\leq 1<p$, by (\textbf{H6}). Thus $\cH_{\kappa,p}(\gamma)\subset \cH_p(\gamma)$. 

Our next theorem shows that assumption \eqref{bound}  holds
for the class of observables in $\cH_{\kappa,p}(\gamma)$. 
Recall $\log^+_{a}b =\max\{0, \log_a b\}$.
\vspace{.2cm}

\begin{theorem}\label{main1}
Assume \textbf{(H1)--(H6)} hold. Let $p>1$,
$0<\kappa\leq p$, and $f,g\in \cH_{\kappa,p}(\gamma)$.
Then  there exist a  constant $\vartheta\in(0,1)$ depending on $ p$ and $\gamma$, such that \begin{align}\label{inequ.f.g}
 |\Cov_n(f,g)| \leq  
 [ 2C_{g,\eps_p}(f) + 4(2C(\gamma, s)C_{\bc}+c_M+C_{\bp}) \|f\|_{p,\gamma}\|g\|_{p.\gamma} ] \vartheta^{n},
\end{align}
for any $n\ge 1$,
where the constant $C_{g,\eps_p}(f)$ is defined as in (\ref{regularcase}).

In particular, there exists a uniform constant $C>0$ such that for any $f,g\in \cH_{\infty}(\gamma)$, we have
\begin{align}
|\int_{M}f\circ T^n\cdot g\,d\mu-\mu(f)\cdot\mu(g)|\leq C\, \|f\|_{\infty,\gamma}\|g\|_{\infty,\gamma}\vartheta^n
\end{align}
for any $n\geq 1$.
\end{theorem}

\begin{remark}
The advantage of our result is that $f$ and $g$ are both in $L^p(\mu)$, comparing to the results on $f\in L^p(\mu)$ and $g\in L^q(\mu)$, with $\frac{1}{p} + \frac{1}{q} =1$.
This is important for us to tackle the case $g=f$ in
the study of the decay rates of autocorrelations for some functions $f\in\cH_{\kappa,p}(\gamma)$ for
$p\in (1,2)$. Note that $ C_0(f,f)=\mu(f^2)-\mu(f)^2$ may diverge since such functions may
have infinite variance.\\
\end{remark}

Next we consider a special  observable -- the index function $R: M\to\mathbb{N}$ defined as in (\ref{returntime}),
where  $R(x)=n$ for any $x\in D_n$ and $n\geq 1$. Then we have the following theorem.
For convenience, we set $C_{0} =2\eps_p^{-1}C_{\bp}(r_0) C_A^{r_0}$, using (\ref{regularcase}).
\begin{theorem}\label{themR}
 Assume \textbf{(H1)--(H6)} hold. Then
the index function $R\in
L^p(\mu)$ for any $p\in (0,\bs+1)$. Moreover, for any $n\geq 1$,
\begin{align}
| \Cov_n( R,  R)|\leq  C  \|R\|_p^2]  \vartheta^n,
\end{align}
where $C$ and $\vartheta\in (0,1)$ are
uniform constants independent of the choices of $n$.
\end{theorem}
\begin{proof} 
Let $p\in (0, \bs+1)$.
 Note that $$\mu(R^p)=\sum_{n\geq 1} n^p \mu(D_n)\leq  C_0 \sum_{n\geq 1} n^{p-2-\bs},$$
where we have used (\ref{alpha1})
 in the last step. Thus $R\in L^p(\mu)$. 
 It follows that $R\in \cH_{\kappa, p}(1)$, with $\gamma=1$, $\kappa=p$, $K_R=1$ in (\ref{Hkappap}). 
 Moreover, $\|R\|_{p,\gamma}=\|R\|_p$, and $N_R=1$.
Theorem \ref{main1} implies that then  there exist a  constant $\vartheta\in(0,1)$ such that
 \[
 |  \Cov_n (R,R)| \leq  [2C_{R,\eps_p}(R) + 4(2C(1, s)C_{\bc}+c_M+C_{\bp}) \|R\|_p^2] \vartheta^n,
 \]
for any $n\ge 1$. This completes the proof.
\end{proof}

\subsection{Applications to induced maps}
\label{sub.app}

Let $\cF:\cM\to\cM$ be a nonuniformly hyperbolic system which preserves a smooth measure $\mu_{\cM}$.  Assume that there exists a
subset $M\subset \cM$ such that the induced system $(M, T, \mu)$ satisfies the assumptions \textbf{(H1)--(H5)},
where $\mu$ is the conditional measure of $\mu_{\cM}$ on $M$, and $T=\cF^R$ is the first return map on $M$. Recall that the first return time function $R: M \to \bN$ is defined by
\beq\label{defnR} 
R(x)=\inf\{n\geq 1:\ \cF^n(x)\in M\}.
\eeq

 Suppose that there exists $n_0\ge 1$ such that for each $n\ge 1$,
 the subset $M_n=\{x\in M: R(x) =n\}$ has at most $n_0$ connected components.
 Let $\{D_{n_0(n-1)+j}:1 \le j \le n_0\}$ be an enumeration
 of  the connected components of $M_n$, $n \ge 1$. Here some of the sets $D_{n}$ can be empty.
 We will call $\{D_m: m\ge 1\}$ the derived enumeration of the dynamical system $(M, T)$.
Suppose \textbf{(H6)} holds for the derived enumeration $\{D_m: m\ge 1\}$.
Then for any H\"older continuous (hence bounded) function $\hat f$ on $\cM$,
it induces a function $f$ on $M$ via
 \[
 f(x)=\sum_{m=0}^{R(x)} \hat f(\cF^m x).
 \]
Generally speaking, the function $f$ is unbounded, since $R$ is.
This { inducing} scheme is an important strategy
in the study of the dynamics of the nonuniformly hyperbolic systems.
Note that $f(x)\leq \|\hat f\|_{\infty} n$  for any $x\in R^{-1}(n)$ and $n\ge 1$.
Therefore, $f\in \cH_{\kappa,p}(\gamma_0)$ with $\kappa=p$.
\vspace{.2cm}

\begin{theorem}\label{main4}
Let $M\subset\cM$, and $T=\cF^R$ be the first return map on $M$,
$\mu$ be the conditional measure of $\mu_{\cM}$ on $M$,
and $\{D_n\}$ be the derived enumeration
of the system $(M, T)$.
Suppose \textbf{(H1)--(H6)} hold, and the first return time function $R\in
L^p(\mu)$ for some $p>1$.
Then for any  $\hat f,\hat g\in C^{\gamma}(\cM)$ with $\gamma\in[\bgamma_0,1]$,
we have
$$| \Cov_n( f,  g)|\leq  C\|\hat f\|_{\infty,\gamma}\|\hat g\|_{\infty,\gamma}\vartheta^n,$$
for any $n\geq 1$,
where $C>0$ and $\vartheta\in (0,1)$ are
uniform constants independent of the choices of $n$.
\end{theorem}

Note that Theorem \ref{main4} follows directly from Theorem
\ref{main1},  with $\kappa=p$ in (\ref{Hkappap}). The proof is similar to that of Theorem \ref{themR}, thus we omit it here.

We also obtain the following results by taking $\hat f=\hat g=1$ in the above theorem \ref{main4}.
\begin{theorem}\label{main6}
Let $(\cM,\cF)$ be a nonuniformly hyperbolic dynamical system and $(M,T,\mu)$ be the induced system on $M$ discussed in the above, which satisfies the Assumptions   \textbf{(H1)--(H6)}.  Then
 there exist  constants $C>0$ and $\vartheta\in (0,1)$ such that
\[
\Big|\int R\circ T^n\cdot R\, d\mu - \mu(R)^2\Big|  \leq C \vartheta^n
\]
for every  $n\ge 1$.
\end{theorem}

As primary examples of the nonuniformly hyperbolic systems, we will discuss several important billiard systems, including two types of billiards: \\

\textbf{I. Type A billiards}:
\begin{itemize}
\item[(A1)]  billiards with
cusps, see
\cite{CM07, CM, CZ08}.
 \item[(A2)] semi-dispersing billiards on a rectangle, billiards with
flat points, see
\cite{CZ05b, CZ08}.
\end{itemize}

\textbf{II. Type B billiards}:
\begin{itemize}
\item[(B1)] Bunimovich stadium, see \cite{Bu74, Bu79, Ma04, CZ08};
\item[(B2)] Bunimovich skewed stadium, Bunimovich flower billiard, see
\cite{ CZ05a, CZ08}.
 \end{itemize}
 All of these models mentioned in the above have been
proved to have only polynomial mixing rates. However, we can show that the autocorrelations $ \Cov_n(R,R)$ for these systems decay exponentially.  Let $(\cM,\cF)$ be the billiard map related to any of the
above models,  and $\mu_{\cM}=C\cos\varphi\, dr\, d\varphi$ be the smooth measure
preserved by $\cF$. It was proved in above references that there exists a
subset $M\subset \cM$ such that the induced system $(M,T,\mu)$ satisfies the
conditions \textbf{(H1)--(H6)}.
It follows that the autocorrelations $ \Cov_n(R,R)$, $n\ge 1$ decay exponentially for all these billiard systems:
\begin{theorem}\label{main6-cor}
Let $(\cM,\cF)$ be one of the type A or type B billiard systems. Let
$(M,T,\mu)$ be the induced system on the corresponding subset $M$. Assume the first return time function $R\in
L^p(\mu)$ for some $p>1$.
Given two functions $\hat f,\hat g\in C^{\gamma}(\cM)$ with $\gamma\in[\bgamma_0,1]$,
then    for any $n\geq 1$,
$$| \Cov_n( f,  g)|\leq  C\|\hat f\|_{\infty,\gamma}\|\hat g\|_{\infty,\gamma}\vartheta^n,$$ where $C>0$ and $\vartheta\in (0,1)$ are
uniform constants independent of the choices of $n$.
\end{theorem}

\section{Preliminaries}\label{sec:4}

In this section  we first introduce the concepts of standard pairs and standard families, then give
a brief review of the growth lemma for hyperbolic systems with singularities.

\subsection{Standard pairs and $\cZ$ functions}\label{stan_pair}

For any unstable manifold $W\in \cW^{u}$, the {\it u-SRB
density} $\rho_W$ is the unique probability density function on $W$
with respect to the Lebesgue measure $m_W$ satisfying
the following relation
 \beq\label{dens}
 \frac{\rho_{W}(y)}{ \rho_{W}(x)} =
\lim_{n\rightarrow \infty}\frac{J_{W}T^{-n}(y)} {J_{W}T^{-n}(x)}.
 \eeq
Then the corresponding probability measure $\mu_{W}:=\rho_W\cdot m_W$ is
called the {\it u-SRB measure } of $T$ on $W$. The formula (\ref{dens}) is  standard  in hyperbolic
dynamics, see \cite[page 105]{CM}.
 It follows from the distortion bound (\ref{distor0}) that $\rho_W\sim |W|^{-1}$ on $W$. More precisely, we have
\beq\label{rhoWbd}
\frac{1}{|W|} e^{-C_{\bJ}|W|^{\bgamma_0}}\leq \rho_W(x)\leq \frac{1}{|W|} e^{C_{\bJ}|W|^{\bgamma_0}}.
\eeq

\begin{definition}(Standard Pair)
Let $C_{\br}>C_{\bJ}$ be a fixed large constant, where $C_{\bJ}$ is the
constant in Eq. (\ref{distor0}) in the assumption of the distortion bounds. For
any unstable curve $W$ and a  probability measure $\nu$ on $W$, the pair
$(W,\nu)$ is said to be a \textit{standard pair}, if $\nu$ is absolutely
continuous with respect to the u-SRB measure $\mu_W$ whose density function
$g(x):=d\nu/d\mu_W (x)>0$ satisfies
 \beq\label{lnholder0}
 | \ln g(x)- \ln g(y)|\leq C_{\br}\cdot d_W(x,y)^{\bgamma_0},
 \eeq
for any $x, y\in W$,
where $\bgamma_0\in (0,1)$ is given in the distortion bounds (\ref{distor0}).  We also call the density function $g$ to be a \emph{dynamically H\"older function} on $W$.
\end{definition}
\vspace{.2cm}

One advantage of utilizing the standard pair is that, the probability measure on any unstable manifold $W$ is
indeed equivalent to the u-SRB measure  $\mu_W$.

The idea of standard pairs was first brought up by Dolgopyat in
\cite{D01}, and then extended to more general systems by Chernov and
Dolgopyat in \cite{C06,CD}.
Note that for any standard pair $(W, \nu)$, the push-forward measure
$T_*\nu$ is a measure supported on countably many unstable curves
$TW=\cup_{n\geq 1} V_{n}$. Thus the image $T(W,\nu):=(TW,T_*\nu)$ can be
viewed as a
family of standard pairs $(V_n, \nu_n)$, $n\ge1$, with
$\nu_n=\frac{T_*\nu|_{V_n}}{T_*\nu(V_n)}$, such that for any measurable set $B\subset
M$,
$$T_*\nu(B)=\sum_n T_*\nu(V_n) \cdot \nu_n(V_n\cap B).$$
This observation leads to the concept of the so called \emph{standard
family},
whose precise definition is given in the following.

\begin{definition}\label{defn3}
Let $\cW=\{(W_{\alpha}, \nu_\alpha)\,:\, \alpha\in \cA\}$ be a
family of standard pairs, $\lambda$ be a Borel  measure on $\cA$.
Then $\cG=(\cW,\lambda)$ is said to be a {\it standard family},
if the following conditions hold:
\begin{enumerate}
\item[(a)] Either $\cA$ is countable, or $\cA\subset [0,1]$ is measurable,
 and $\{W_{\alpha}: \alpha\in \cA\}$ is  a measurable partition
    of a measurable set $E\subset M$;

\item [(b)] The measure $\lambda$ induces a positive measure $\nu$  on $M$  by
\[
   \nu(B)=\int_{\alpha\in\cA} \nu_{\alpha}(B\cap
   W_{\alpha})\,
   d\lambda(\alpha),\hspace{1cm}
   \]
for all measurable sets $B\subset M$.
\end{enumerate}
\end{definition}

In the following, instead of the notion $(\cW,\lambda)$, we also denote a standard family $\cG$ by $(\cW, \nu)$, or
$(\cW,\cA,\lambda,\nu)$,  since we use the measure $\nu$
more often.
An intuitive way to look at the standard family is to view it as a
decomposition of the measure $\nu$ along leaves of a measurable partition
$\cW=\{W_{\alpha}, \alpha\in \cA\}$ of the set $E$ into standard pairs.
\vspace{.2cm}

It follows from \cite[Lemma 12]{CZ09} that
 $\cE:=(\cW^u,\mu)$ can be viewed as a standard family,
where $\mu$ is the SRB measure of $T$. More precisely,
there exists a factor measure $\lambda^u$ on the index set $\cA^u$,
such that for any measurable set $A\subset M$,
 \beq\label{muMdisintegrate}
\mu(A)=\int_{\alpha\in \cA^u} \mu_{\alpha}(A\cap
W_{\alpha})\,d \lambda^u(\alpha),
 \eeq
where $\mu_{\alpha}=\mu_{W_{\alpha}}$ is the $u$-SRB measure on the unstable
manifold $W_{\alpha}$. From now on we always take $\mu_{\alpha}$ to be the reference measure on
the unstable manifold $W_{\alpha}\in \cW^u$.
Sometime we also denote it as $\mu_W$ for a general unstable manifold
$W\in \cW^u$.
Note that a probability measure $\nu$ with $\nu\ll\mu$ is uniquely
specified by its density
$g:=d\nu/d\mu\geq 0$.
Then comparing with (\ref{muMdisintegrate}) for $\mu$,
we know that for any standard family $(\cW^u,\nu)$ with $\nu\ll\mu$,
 \beq\label{nugalpha}
 \nu(A)=\int_A g\,d\mu
=\int_{\alpha\in \cA^u} \int_{W_{\alpha}\cap A}
g\,d\mu_{\alpha}\,d\lambda^u(\alpha)
=\int_{\alpha\in \cA^u}\int_{W_{\alpha}\cap A}
d\nu_{\alpha}\,d\lambda(\alpha),
 \eeq
for any measurable set $A\subset \cM$, where
\beq\label{nuWg}d\nu_{\alpha}=g_{\alpha}\,d\mu_{\alpha}\,\,\,\,\,\,\,\,\text{  with }
\,\,\,\,\,\,\,\,g_{\alpha}=\frac{g}{\mu_{\alpha}(g)}\ \ \mbox{and} \ \ d\lambda(\alpha)=\mu_{\alpha}(g) d\lambda^u(\alpha).\eeq
This is called a {\it canonical representation} of $(\cW^u,\nu)$.

On the other hand, consider
a nonnegative function $g: M\to [0,\infty)$, with $\mu(g)<\infty$, such that it is {\it{dynamically H\"older continuous on $\cW^u$}}, i.e.:
\beq\label{lnholder01}
 | \ln g(x)- \ln g(y)|\leq C_{\br}\cdot d_{W}(x,y)^{\bgamma_0},
 \eeq
 for any $W\in \cW^u$, and any $x, y\in W$,  with $g(x), g(y)>0$.
We
call $(\cW^u,\nu_g)$ {\it the standard family generated by $g=d\nu_g/d\mu$}, such that (\ref{nuWg}) holds.

\vspace{.2cm}

The assumption \textbf{(H3)} on the distortion bound implies that, if $\cG=(\cW, \nu)$ is
a standard family with a factor measure $\lambda$ on $\cA$, then
$T^n\cG:=(T^n \cW, T^n \nu)$ also induces a standard family,
for any $ n\geq 0$ (cf. \cite{CZ09}). More precisely, for any measurable set $B\subset M$,
\[
   T^n \nu(B)=\int_{\alpha\in \cA}
   T^n \nu_{\alpha} (B\cap T^n W_{\alpha}) \,d\lambda(\alpha).
\]

Note that the (uniform) hyperbolicity of $T$ only
guarantees the exponential growth of unstable manifolds in a local sense. The
singularity curves will cut any {\it relatively long} unstable manifold into many short pieces,
which might undo the uniform growth of the unstable manifold completely.
However, under the one-step expansion assumption (\textbf{H5}), one can
show that, despite odds, typical smooth components of $T^n W$ for every short
unstable manifold $W$ grow monotonically and exponentially in $n$, until they
reach the size of order one. This fundamental fact follows from the {\it
Growth Lemma}, which was first proved by Chernov for dispersing billiards
(cf. \cite{C99}), and then for general hyperbolic systems under the assumption
\textbf{(H1)--(H5)} in
\cite{CZ09} (see Lemma 5, Lemma 6 and Lemma 12 therein).
To state the growth lemma, we need to introduce a characteristic
function of $\cG$.

\begin{definition}
The $\cZ_r$-function $\cZ_r(\cG)$ of a standard family $\cG=(\cW,\nu)$ is given by:
 \beq\label{cZ}
\cZ_r(\cG):=\frac{1}{\nu(M)}\int_{\cA}|W_{\alpha}|^{-r}\,d\lambda(\alpha),
\quad r\in (0,\bs_0].
 \eeq
\end{definition}

Note that $\cZ_r(\cG)$ can be viewed as an average of the inverse of the
$r$-power of the length of the unstable manifolds. 
This $\cZ_r$-function plays an important role in estimating the
distribution of short unstable manifolds in a standard family,
and the growth lemma characterizes how
$\cZ_r(T^n\cG)$ changes under iterations of the map $T$.

 A standard family $\cG$ is said to be
\emph{$r$-standard} if $\cZ_r(\cG)<\infty$.  {The following is a collection of
Growth Lemmas from \cite{CZ09}.}
\begin{lemma}\label{growth}
Let $r\in (0,\bs_0]$, and $\cG=(\cW,\nu)$ be an $r$-standard family. Then the following statements hold.
\begin{enumerate}
\item[(1)] There exist constants $c>0$, $C_z>0$, and $\gamma_1\in (0,\bgamma_0]$
    such that  \beq\label{firstgrowth}
     \cZ_{r}(T^n \cG)\leq c\cZ_r(\cG)\cdot\Lambda^{-\gamma_1 n}
     +C_z.
     \eeq
\item[(2)] { For any $\eps >0$, we have}
\beq\label{standard}\nu\{W\in \cW:\
|W|<\eps\}\leq c\nu(M)\cdot \cZ_r(\cG)\cdot\eps^r.
\eeq
\item[(3)] The standard family $\cE=(\cW^u, \mu)$ induced by the SRB measure
$\mu$ is {$r$-standard for any $r< \bs_0$}.
\end{enumerate}
\end{lemma}
{For more details, see Lemma 7, Lemma 8 and Lemma 12 in  \cite{CZ09}.
Note that $\bs_0$ from \eqref{onestep} plays the same as $q$ from \cite{CZ09}.
The letter $q$ has been reserved for later use. }

Let $r\in (0, \bs_0)$ be fixed. Pick  a large 
number\footnote{The choice of the constant $C_{\bp}$ does not matter as long as it is large enough.
A stronger condition is given in \eqref{Cp2} using  the constant in Lemma  \ref{convergerp}.} 
\beq\label{Cp1}C_{\bp}> \frac{C_z}{1-c} + \cZ_r(\cE).\eeq
Then an $r$-standard family $\cG$ is said to be \emph{$r$-proper} if
$\cZ_r(\cG)<C_{\bp}$.
 It follows from this definition that:
\begin{lemma}\label{properagain}
\begin{itemize}
\item[(1)] The standard family $\cE=(\cW^u, \mu)$ is {$r$-proper for any $r< \bs_0$}.

\item[(2)]The forward iterations  of $r$-proper families are $r$-proper.

\item[(3)] For any $r$-standard family $\cG$, the family $T^n \cG$ is
    $r$-proper for each
    $n \ge N_{r,\cG}:=\lceil \frac{1}{\bgamma_1}\log_{\Lambda}^{+} \tfrac{\cZ_r(\cG)}{C_{\bp}} \rceil$.
\end{itemize}
\end{lemma}
\begin{proof}
(1) Let  $r< \bs_0$ be given. It follows from Lemma \ref{growth} that $\cE$ is $r$-standard and
\begin{align*}
\cZ_r(T^n\cE) \le c\cZ_r(\cE)\cdot \Lambda^{-\gamma_1 n} +C_z
< \cZ_r(\cE) +C_z < C_{\bp}
\end{align*}
for each $n\ge 0$. Therefore, $\cE=(\cW^u, \mu)$ is $r$-proper.

(2) Let $\cG$ be an $r$-proper family and $n\ge 1$ Then  by  \eqref{firstgrowth},
\begin{align*}
\cZ_r(T^{n}\cG) \le c\cZ_r(\cG)\cdot \Lambda^{-\gamma_1 n} +C_z
<c\cdot C_{\bp} +C_z < C_{\bp}.
\end{align*}
Therefore, $T^n\cG$ is $r$-proper.

(3) Let $\cG$ be an $r$-standard family and
$n\ge \frac{1}{\bgamma_1}\log_{\Lambda}^{+} \tfrac{\cZ_r(\cG)}{C_{\bp}}$ be an integer.
Then by  \eqref{firstgrowth},
\begin{align*}
\cZ_r(T^{n}\cG) \le c\cZ_r(\cG)\cdot \Lambda^{-\gamma_1 n} +C_z< C_{\bp}.
\end{align*}
 Therefore, $T^n\cG$ is $r$-proper.
\end{proof}

\subsection{The standard family associated to a density function in $\cH_p(\gamma)$.}In this section, we assume $p\in(1,\infty]$ and $\gamma\in[\bgamma_0,1]$, $g\in \cH_{p}(\gamma)$ is
a nonnegative, dynamically H\"older function with $0<\mu(g)<\infty$.
We have showed that $g$  induces a standard family $\cG_{g}=(\cW,\nu)$, with $d\nu=gd\mu$. Now we will investigate the $\cZ$-function of $\cG_{g}$, and show that $g$
leads to an $r$-standard family with $r\leq \bs_0$.
\begin{lemma}\label{rpropergg}
Let $p>1$, and $g\in \cH _{p}(\gamma)$ be a probability density function. 
Define $\bar g(\alpha)=\mathbb{E}(g|W_{\alpha})$  for any $\alpha\in \cA^u$. 
Then for any $r\in (0, \bs_0- \frac{\bs_0}{p})$,
\begin{enumerate}
\item[(1)]  
$\cG_{\bar g}$ is $r$-standard and  
$\cZ_r(\cG_{\bar g})\leq  C_{\bp} \|g\|_{p,u}$.
\item[(2)] $\cG_{g}$ is $r$-standard and 
$\cZ_r(\cG_{g})\leq  C_{\bp}( \|g\|_{p,u}+\|g\|_{\gamma})$.
\item[(3)] $T^{n}\cG_{g}$ is  $r$-proper
for $n\geq N_{g}:= \lceil \frac{1}{\gamma_1}\log_{\Lambda}^{+}(\|g\|_{p,u}+\|g\|_{\gamma}) \rceil$. 
In particular, $\cG_{g}$ is $r$-proper if  $\|g\|_{p,u}+\|g\|_{\gamma}\leq 1$.
\end{enumerate}
\end{lemma}
\begin{proof}

(1). 
Note that if $\cG_{\bar g}=(\cW,\nu_{\bar g}))$ is generated by the conditional expectation of a probability density function $\bar g$, then (\ref{nugalpha}) implies that
\beq\label{cZg}
\cZ_r(\cG_{\bar g})=\int_{\cA^u}\frac{\bar g}{|W_{\alpha}|^{r}}\,d\lambda^u(\alpha),
\quad r\in (0, { \bs_0}],
 \eeq

 The fact that $\cE=(\cW^u,\mu)$ is $r_1$-proper, for any $r_1\in (0,\bs_0)$, implies that
$$\cZ_{r_1}(\cE)=\int_{\cA^u}\frac{1}{|W_{\alpha}|^{r_1}}\,d\lambda^u(\alpha)\leq C_{\bp}.$$

By definition, we have
 $$\|g\|_{p,u}=\sup_{\alpha\in \cA^u}  |g(x)| |W_{\alpha}|^{s_0/p}<\infty.$$
 For any $r<\bs_0-\bs_0/p=\bs_0/q$, there exists $r_1<s_0$, such that $r\leq r_1-\bs_0/p$. 
 It implies that
 $$\cZ_r(\cG_{\bar g})=\int_{\cA^u}\frac{1}{|W_{\alpha}|^{r_1}}\, \bar g(\alpha) |W_{\alpha}|^{r_1-r}\,\cdot d\lambda^u(\alpha)<\|g\|_{p,u} C_{\bp}.$$

(2). For any $x\in W_{\alpha}$, $\alpha\in\cA^u$,
\beq\label{bargg}
|g(x)-\bar g(\alpha)|\leq \|g\|_{\gamma}\cdot  |W_{\alpha}|^{\gamma}.
\eeq
Then (\ref{nugalpha}) implies that
 \begin{align}\label{cZgbarg}
&|\cZ_{r}(\cG_{g})-\cZ_{r}(\cG_{\bar g})|
\leq  \|g\|_{\gamma}\sum_{n\ge 1}  \int_{\cA_n}\frac{|W_{\alpha}|^{\gamma}}{|W_{\alpha}|^{r}}\,d\lambda(\alpha)
\leq \|g\|_{\gamma}\cZ_r(\cE).
 \end{align}
Combining with item (1), we get
$$\cZ_r(T\cG_{{g}})\leq C_{\bp}(\|g\|_{p,u} +\|g\|_{\gamma}).$$

In the case $g\in \cH_{\infty}(\gamma)$, it is evident that
$\cZ_{r}(\cG_{\bar g})\leq \|g\|_{\infty} \cZ_{r}(\cE)<\infty$. Thus $\cG$ is $r$-standard, for any $r<\bs_0$.

(3). The statement directly follows from (\ref{firstgrowth}).

\end{proof}

Note that above lemma implies that for any $p>1$ and for any dynamically H\"older function
$g\in \cH_p(\gamma)$ with $\mu(g)<\infty$, the associated standard family
$\cG_{g}$ is $r$-standard for any $r<\bs_0-\bs_0/p$.

 %%% end of that subsection

\section{The Coupling Lemma and the  hyperbolic magnet}\label{magnet}

In this section, we will construct a hyperbolic set $\cR^*$,
the so called {\it magnet} (see \cite{CM}), and build a  hyperbolic suspension structure
for $(T,\mu)$.
The hyperbolic set $\cR^*$
will be used later as the base for our coupling algorithm.

\subsection{The construction of a hyperbolic magnet}

Given a family of aligned stable manifolds $\Gamma^s$ and a family of aligned
unstable manifolds $\Gamma^u$ that fully across each other, their intersection  $\cR=\Gamma^u\cap\Gamma^s$
is said to have  a {\it hyperbolic product structure}. A subset
$\cR_1\subset\cR$ is said to be a $u$-subset  of $\cR$, if there exits a subfamily of
unstable manifolds $\Gamma^u_1\subset \Gamma^u$, such that
$\cR_1=\Gamma^u_1\cap\Gamma^s$. Similarly we can define the $s$-subsets of $\cR$.

Now we begin to construct this hyperbolic set $\cR^*$. We first define the radius of an invariant manifold:
\vspace{.2cm}

\begin{definition}
For any $x\in M$, let $r^{s/u}(x)=d_{W^{s/u}(x)}(x,\partial W^{s/u}( x))$,
where $W^{s/u}(x)$ is the stable/unstable manifold that contains $x$. Set
$r^{s/u}(x)=0$ if the stable/unstable manifold at $x$ degenerates.
\end{definition}

For every
$\delta>0$, we introduce the following set
 \beq\label{Ndelta}
N_{\delta}^{\pm}=\{x\in M\,:\, d(T^{\pm n}x,  S_{\pm 1})\geq
\delta\Lambda^{-n} \text{ for all } n\ge0\}.
 \eeq
Note that the complement of the set $N_{\delta}^{\pm}$ contains points whose
orbits approach $ S_{\pm 1}$ under $T^{\pm n}$ faster than the
$\delta\cdot\Lambda^{-n}$ for some $n$. It is well known that a point $x$ can
not have long stable/unstable manifold if the orbit of $x$ approaches to the
singularity set $ S_{\pm {1}}$ too fast under $T^{\pm n}$.
The following lemma states that the stable manifolds (resp. unstable manifolds)
of points in $N_{\delta}^{+}$ (resp. $N_{\delta}^{{-}} $) have a uniform length.
\vspace{.2cm}

\begin{lemma}[\cite{CZ09}]\label{global}
Assume (\textbf{H1)--(H5}) hold. There exists a uniform constant $c>0$, such that
for any $\delta>0$,  $r^{s}(x)>c\delta$ for every point $x\in
N_{\delta}^{+}$ and $r^{u}(x)>c\delta$ for every $x\in N_{\delta}^{-}$.
\end{lemma}

It follows from Lemma \ref{global} that
for any point $x\in N_{\delta}:=N_{\delta} \cap N_{\delta} $,
both $W^s(x)$  and $W^u(x)$ exist with $r^s(x)\geq c\delta$
and $r^u(x)\geq c\delta$. Note that $x$ may not belong to
$N^{\pm}_\delta$ even if $r^{s/u}(x)\geq c\delta$,
since $\Lambda^{-n}$ is just an upper bound, and $T^nx$ may approach
to $S_{\pm1}$ faster than that.
\vspace{.2cm}

The next result is also proved in \cite{CZ09}, whose proof  depends strongly on (\textbf{H4}).
\begin{lemma}\label{growth2}
There exists $\delta_0>0$, such that for any unstable manifold $W^u$ with
$|W^u|>\delta_0$,
 \beq\label{shortsu}
 \mu_{W^{u}}(r^s(x)<\eps)<C \eps^{\bs_0}, \text{ for any }\eps>0,
 \eeq
where $C>0$ is some uniform constant depending only on $\delta_0$.
\end{lemma}

\begin{remark}
The relation in (\ref{shortsu}) has a `time reversal' counterpart, see the
remark after  \cite[Theorem 5.66]{CM}. These two equations guarantee that
there are plenty of ``long" stable (resp. unstable) manifolds along any
unstable  (resp. stable) manifolds, which is essential for the construction of
the hyperbolic magnet.
\end{remark}

\begin{remark}
Using (\ref{standard}) and the fact that the SRB
measure $\mu$ is $\bs_0${-}proper,  we know that that
$\ds \mu (M\backslash N^{\pm}_\delta)\leq C \delta^{\bs_0}$.
Therefore, the set of points with short stable/unstable manifolds
satisfies the following estimate: there exists $\tc >0$ such that
$\mu(r^{s/u}(x)<\delta)\leq \tc \delta^{\bs_0}$ for all $\delta>0$.
\end{remark}

Now let $d\in(0,1)$, pick $\delta_0$ small enough such that
$\mu(N_{10\delta_0/c})>d$. Pick an unstable manifold $W$ with $\mu_W(W\cap
N_{10\delta_0/c})>d$, and a $\mu_W$-density point $x_0\in W\cap
N_{10\delta_0/c}$. Define
\begin{equation}\label{hatgammas}
  \hat\Gamma^s=\{ W^s(y)\,|\, y\in W^u(x_0)\cap N_{10\delta_0/c}\}.
\end{equation}
Note that $\hat\Gamma^s$ is the collection of all maximal stable manifolds
along $W^u(x_0)\cap N_{\delta_0}$, which stick out both sides of $W^u(x_0)$
by at least $5\delta_0$. As  the length of stable manifolds in $\hat\Gamma^s$
may be very irregular, we need to chop off a portion to get our magnet.
\vspace{.2cm}

Let $\cU$ be a  {\it rectangular} shaped region such that $x_0$ is
almost the geometric center of $\cU$, and the boundary
 $\partial \cU$ consists of two unstable manifolds with
length $6\delta_0$ and two stable manifolds with length
$6\delta_0$. Accordingly,  the region $\cU$ can be viewed as a
rectangle centered at $x_0$ with dimensions
$6\delta_0\times 6\delta_0$.
\vspace{.2cm}

We say that an unstable manifold $W^u$
 \emph{fully u-crosses} $\hat\Gamma^s$, if
$W^u$ meets every stable manifold in $\hat\Gamma^s$.  Let $\hat\Gamma^u$ be
the collection of all maximal unstable manifolds $W^u(y)$ that fully $u$-cross
$\hat\Gamma^s$, with $y\in W^s(x_0)\cap\cU$.
Let $\Gamma^{u/s}=\hat\Gamma^{s/u}\cap\cU$, and then set
$\cR^*=\hat\Gamma^s\cap\hat\Gamma^u$.
By the bounded distortion property of the stable holonomy map $\bh$,  we get that there
exists $d_0>0$ such that
 \beq\label{Wd0}
 \mu_{W}(W\cap \cR^*)>d_0,
 \eeq
 for any $W\in \Gamma^u$.
Moreover, we have $\mu(\cR^*)>d_0^2$.
\vspace{.2cm}

\subsection{The Coupling Lemma}

 Next we review the {\it Coupling Lemma}, which was originally proved by Chernov and Dolgopyat
(cf. \cite{C06, CD, D01}) for dispersing billiards,
see also \cite[\S 7.12--7.15]{CM}.
The coupling lemma was
generalized in \cite{CZ09} to systems under the assumptions \textbf{(H1)--(H5)}
on proper families, and
then in \cite{SYZ} to time-dependent billiards.
\vspace{.2cm}

When performing the coupling algorithm on a given magnet $\cR^\ast$, one may
not couple the entire measure crossing the magnet at each return time. To
implement the idea, we use the concept of the {\it generalized standard family}.

 \begin{defn}\label{pseodugs}
 Let $(\cW,\nu)$ be a standard family, where $\cW\subset \Gamma^u$ is a measurable collection of unstable manifolds in $\Gamma^u$. Then we define $(\cW,\nu)|_{\cR^*}:=(\cW\cap \cR^*, \nu|_{\cR^*})$. For any $n\geq 0$, we call $(\cW_n,\nu_n):=T^{-n} ((\cW,\nu)|_{\cR^*})$ as a generalized standard family with index $n$.
 \end{defn}

Next we state  the Coupling Lemma \cite{CD,CM} for the induced
system $(F, \mu_M)$ using the concept of generalized standard families.

\begin{lemma}\label{coupling1}
Let  $\cG^i=(\cW^i, \nu^i)$, $i=1, 2$,  be two  $r$-proper families on $M$,
for some $r\in(0,\bs_0]$. For any $n\geq 1$,
 there exist two sequences of  generalized standard families $\{(\cW^i_n,\nu^i_n), n\geq 0\}$,  such that
  \beq\label{decomposeGi}
  \cG^i=\sum_{n=0}^{\infty}(\cW^i_n,\nu^i_n):=\Big(\bigcup_{n=0}^{\infty}\cW^i_n,\sum_{n=0}^{\infty}\nu^i_n\Big).
 \eeq
Moreover, they satisfy the following properties that for each $n\geq 0$:
\begin{itemize}
\item[(i)] \textbf{Proper return to $\cR^*$ at $n$.}\\
 \,\,\,\,\, Both $(\cW^1_n,\nu^1_n)$ and $(\cW^2_n,\nu^2_n)$ are  generalized standard families of index $n$;
  \item[(ii)] \textbf{Coupling $T^n_*\nu^1_n$ and $T^n_*\nu^2_n$ along stable manifolds in $\Gamma^s$.}\\
   \,\,\,\,\,\,\,For any measurable collection of stable manifolds $A\subset \Gamma^s$, we have $$T^n_*\nu_n^1(A)=T^n_*\nu^2_n(A).$$
 \item[(iii)] \textbf{Exponential tail bound for uncoupled measure at $n$.}\\
     \,\,\,\,\, There exist $C_{\bc}>0$ and $s \in (0, \bs_0]$ such that, for any $n\geq 1$,
  \beq\label{ctail1}
  \bar\nu_n^i(M)<C_{\bc}\Lambda^{-sn },
  \text{  for all  }n\ge 1,\eeq
  where $\bar\nu_n^i:=\sum_{k\geq n}\nu^i_k$ is the uncoupled measure at $n$-th step.
   \end{itemize}
\end{lemma}

\begin{proof}
Given  two  $r$-proper families $\cG^i=(\cW^i, \nu^i)$, $i=1,2$, on $M$,
let $\widehat{\cW}^i=\cW^i \times [0,1]$, on which the map $T$ can be extended
by $T(x,t) = (Tx, t)$. Let $\hat\nu^i = \nu^i\times dt$ be the corresponding product measure.
By Lemma 15 in \cite{CZ09}, there exist
a coupling  map $\Theta: \widehat{\cW}^1 \to \widehat{\cW}^2$
and a coupling time function $\Upsilon:  \widehat{\cW}^1 \to \bN$
such that
\begin{enumerate}
\item $\Theta_{\ast}(\hat\nu^1) = \hat\nu^2$,

\item for any $(x,t) \in \widehat{\cW}^1$, $T^{n}y \in \Gamma^s(T^n x)$, where $(y,s)=\Theta(x,t)$
and $n=\Upsilon(x,t)$.
\end{enumerate}
For each $n\ge 1$,
let $\widehat{\cW}^1_n=\Theta^{-1}(n)$, $\hat\nu^1_n$ be the restriction of $\hat\nu^1$ on $\widehat{\cW}^1_n$,
$\cW^1_n$ be the projection of $\widehat{\cW}^1_n$ and $\nu^1_n$ be the projection of $\hat\nu^1_n$
on $\cW^1_n$. Then $(\cW^1_n, \nu^1_n)$ is a generalized standard family for each $n\ge 1$.
Similarly we define $\cW^2_n$ and $\nu^2_n$, $n\ge 1$.
Then $T^n\nu^1_n$ and $T^n\nu^2_n$ are coupled by the stable manifolds in $\Gamma^s$, $n\ge 1$.
Let $\bar\nu_n^i:=\sum_{k\geq n}\nu^i_k$, $i=1, 2$. Then \eqref{ctail1} follows from \cite[Eq. (7.2)]{CZ09}.
This completes the proof.
\end{proof}

\vspace{.2cm}

\begin{remark}
Each part $(\cW^i_n,\nu^i_n)$ is supported on some Cantor set
of the unstable manifolds, instead of the whole manifolds.
However, note that the measure $\nu^i_n$
is the restriction of some standard family
to that Cantor set.
This is why we call such a family a {\it generalized} standard family.
\end{remark}

The lemma in the following tells us that, under
that assumptions \textbf{(H1)--(H5)},
the system admits a hyperbolic set with a generalized Markov structure.
\vspace{.2cm}

\begin{lemma}\label{thm:main1}
Suppose the map $T:M\to M$ satisfies
the assumptions \textbf{(H1)--(H5)}.
Then the hyperbolic set $\cR^*$ is the base of a generalized Markov partition with exponential small tails:
\begin{enumerate}
\item[(1)] $\cR^*$ has a decomposition into $s$-subsets $\cR^*=\cup_{n\geq 1} \cR^*_n$
    such that  $T^n\cR^*_n\subset\cR^\ast$ is a $u$-subset of
    $\cR^\ast$ for each $n\ge 1$.

\item[(2)]  There exist $C>0$ and $s\in (0,\bs_0]$ such that
the following holds for each $n\ge 1$:
    \beq\label{cRmbd}
    \sum_{k\geq n} \mu(\cR^*_k)\leq C \Lambda^{-sn}.
    \eeq
\end{enumerate}
\end{lemma}

\begin{proof} [Proof of Lemma \ref{thm:main1}]

We apply  Lemma \ref{coupling1} to two identical copies of the standard family induced by
the SRB measure, and then construct the generalized Young tower based on this
coupling process. That is, let $\nu$ be the restriction of  the SRB measure $\mu$  on $\bigcup \Gamma^u$,
$\lambda^u$ be the factor measure of $\nu$ on $\Gamma^u$,
and $\cG^i:=(\Gamma^u,\nu)=\{(W,\mu_{W})_{W\in \Gamma^u}, \lambda^u\}$, $i=1, 2$,
be two identical copies of the induced standard family.  By Lemma \ref{coupling1}, for each $n\geq 1$,
we couple {\it everything}\footnote{This is due to the fact that the two families are the same. }
in $T^n \cG^i$ that properly return to $\cR^*$,
and get a decomposition  of $(\Gamma^u,\nu)$ into
{\it generalized} standard families $\{ (\cW_n,\nu_n), n\geq 1\}$,
such that $T^n \cW_n$ is a $u$-subset of $\cR^*$ for each $n\ge 1$.
Define $\cR_n=\cW^s_n\cap \Gamma^s$, where $\cW^{s}_n=\{W^s(x)\,:\, T^n x\in \Gamma^s\cap \cF^k\cW_n\}$.
Then we can check that
$$\cR^*=\cup_{n=1}^{\infty}\cR_n\,\,\,\,\,(\text{mod} \,0),$$
and it satisfies the following properties:\\
(1)  $T^{ n}(\cR_n)$  is a $u$-subset of $\cU^*$ and $\{\cR_n, n\geq 1\}$ are almost surely disjoint $s$-subsets of $\cU^*$ in the following sense: $\mu(\cR_m\cap\cR_n)=0$ for any $m\neq n$;\\
(2) Furthermore
   \beq\label{ctail4}\sum_{k=n}^{\infty} \mu(\cR_k)<C_{\bc}\Lambda^{-sn },\eeq
   where $C_{\bc}$ and $s>0$ are from (\ref{ctail1}).

This completes the proof.
\end{proof}

Note that the singularity sets  $S_{\pm 1}$  may contain countably many smooth curves.
Thus there can exist countably many $s$-subsets $\cR_{n,i}$, $i\geq 1$, such that $\cR_n=\cup_{i\geq 1}\cR_{n,i}$, and $T^n W^u(x)\in \Gamma^u $ properly crosses $\cU^*$ for any $x\in\hat\Gamma^s_{n,i}$.
Moreover, $T^n y$ and $T^n x$ belong to the same unstable manifold in $\Gamma^u$ for any $x,y\in \cR_{n,i}\cap W^u$.
\vspace{.2cm}

One can build a partition of $M$ by taking the union
of all forward iterates $T^k\cR^\ast$,
and extend the return time function on $\cR^\ast$ to the whole space $M$.
That is, $\tau(x)=n-k$ as the time needed to {\it enter} the base $\cR^\ast$ of the
tower for any $x\in T^k\cR^*_n$, $0\le k< n$ and
$n\ge1$. It follows from (\ref{cRmbd}) that
$$\mu(\tau>n)\leq C_1 \Lambda^{-s n},$$
for some uniform constant $C_1=C/(\Lambda^{s}-1)$.
This generalized Markov tower based on the
partition $\cR^*=\cup_{n\geq 1}\cR_n$ is in the same spirit of  \cite{Y98,Y99}.
One improvement here is that
we allow the minimal solid $s$-rectangle containing $\cR_n$
to consist of countably many minimal $s$-rectangles  $\cU^*_{n,i}$.
This property is due to the fact that  we allow the singular set of system to contain countably many singular curves,
since one unstable manifold may be cut into infinitely many small pieces, many of which may return to the rectangle $\cU^*$ simultaneously.
So from this point of view, our construction covers more classes of dynamical systems than those \cite{Y98,Y99}, including dispersing billiards with infinite horizon.\\

\section{Decay of Correlations}\label{decay1}

In this section we will study  the decay rates of correlations for the system
$(T, M, \mu)$  using the coupling method.

\subsection{Equidistribution properties}\label{equidistribution}

The following lemma describes the equidistribution property of the system.
Namely, for any two probability measures $\nu^i$, $i=1,2$, associated with
proper families $\cG^i=(\cW^i,\nu^i)$, the images $T^n \nu^i$, $i=1, 2$ approach
each other exponentially (in the week topology for a large set of observables). As
explained
in \cite{CD}, the equidistribution property effectively describes the
asymptotic independence between the present and the future of the system.

Recall  the constants $s=s(\mu,\nu)$ and $C_{\bc}>0$ given by Lemma \ref{coupling1}.
\vspace{.2cm}
\begin{lemma}\label{decayinf}
Let $p>1$, $\gamma\in[\bgamma_0,1]$, and $(\cW, \nu)$ be an $r$-proper family 
 with density $d\nu/d\mu\in\cH_{p}(\gamma)$.  Under condition (\ref{bound}
), then for any $f\in \cH_{p}(\gamma)$,
$$\left|\int_M f\circ T^n \,d\nu-\int_M f d\mu\right|
\leq (2C_{\bc}C(\gamma, s)\|f\|_{p,\gamma}+C_{g,\eps}(f))\Lambda^{-\xi n},$$
where $\xi=\min\big\{\frac{\gamma}{2}, \frac{s}{2}, \frac{s}{q}, \frac{\eps}{1+\eps}s\big\}$,
and $C(\gamma, s) = \frac{1}{1-\Lambda^{-\gamma}}+\frac{1}{1-\Lambda^{-\gamma}}$.
\end{lemma}

\begin{proof} Let $(\cW, \nu)$ be the $r$-proper family,
We first unify the notations by denoting $(\cW^1,\nu^1):=(\cW,\nu)$,
and $(\cW^2,\nu^2):=(\cW^u, \mu)$. By the Coupling Lemma
\ref{coupling1}, we have a decomposition $\nu^i=\sum_{m=1}^n\nu_m^i+\bar\nu^i_n$ with
support $\supp (\nu_m^i)=\cW_m^i$, such that $T^{m} \cW_m^i$ is a
$u$-subset of $\cR$. Note that  the uncoupled measures satisfy
$\bar\nu^1_n(M)=\bar\nu^2_n(M)<C_{\bc}\Lambda^{-sn}$.
\vspace{.2cm}

Now for any $x\in W^u\subset \Gamma^u$, we choose $\bar x\in W^s(x)$, such
that $f(\bar x)=\max_{y \in W^s(x)} f(y)$ be the maximum value of $f$
along stable manifold $W^s(x)$.
\vspace{.2cm}

We first consider the coupled portions $\nu^1_m$ and $\nu^2_m$
when $m\leq n$. In this case the H\"{o}lder continuity of $f$ implies that for
$\mu$-a.e. $x\in M$,
$$|f\circ T^{n-m}(x)-f\circ T^{n-m}(\bar x)|\leq
\|f\|_{C^\gamma}\Lambda^{-\gamma (n-m)}.$$
Then we have,
\begin{align*}
|\int_{ \cW^1_m} f\circ T^n(x) d\nu_m^1(x)&-\int_{ \cW^2_m} f\circ T^n(y) d\nu_m^2(y)|
\leq |\int_{ T^m\cW^1_m} \left(f\circ T^{n-m}(x) -f\circ T^{n-m}(\bar x)\right)dT^m_*\nu_m^1(x)|\\
&+ |\int_{ T^m\cW^2_m} \left(f\circ T^{n-m}(y) -f\circ T^{n-m}(\bar x)\right)dT^m_*\nu_m^2(y)|\\
&+|\int_{ T^m\cW^1_m} f\circ T^{n-m}(\bar x)\ dT^m_*\nu_m^1(x)-\int_{T^m\cW^2_m} f\circ T^{n-m}(\bar x)\ dT^m_*\nu_m^2(y)|\\
&\leq 2C \|f\|_{C^\gamma}\cdot \Lambda^{-\gamma(n-m)}\cdot \mu_m(M) 
\le2C_{\bc}\|f\|_{C^\gamma} \Lambda^{-\gamma(n-m)-s m}.
\end{align*}
Here we use the fact that
$(T^m\cW^1_m,T^m_*\nu_m^1)$ and
$(T^m\cW^2_m,T^m_*\nu_m^2)$ are coupled at time $m$ and have the same measure.
In particular, the choices of $\bar x$ works for both families. 
\vspace{.2cm}

 Recall that  $\bar\nu_n^2$ is the remaining measure of $\nu^2=\mu$ after time $n$.
Then by H\"older
inequality and Lemma \ref{coupling1}, and Lemma \ref{coupling1} (ii), we have for $i=1,2$,
\begin{align*}
\left|\int_{ \bar{\cW}^i_n} f\circ T^n d\bar\nu_n^i\right|=\left|\int_{T^n \bar{\cW}^i_n} f\, dT^n \bar\nu_n^i\right|
&\leq \left(\int_{M} |f|^{p} dT^n_*\bar\nu_n^2 \right)^{\frac{1}{p}}\cdot
\left(T^n\bar\nu_n^2(M)\right)^{\frac{1}{q}}
\leq \|f\|_{p}\cdot C_{\bc}\Lambda^{-s n/q},
\end{align*}
where $q>1$ satisfies $\frac{1}{p}+\frac{1}{q}=1$. Here we use the assumptions
that the measure $\nu^2=\mu$ is $T$-invariant and the function $f\in \cH_p^+(\gamma)$.
\vspace{.2cm}

For the counterpart $\bar\nu_n^1$ of $\nu^1=\nu$, we have
\begin{align}
\left|\int_{ \bar\cW^1_n} f\circ T^n d\bar\nu_n^1\right|
&\leq \left(\int_{M} |f\circ T^n|^{1+\eps} d\bar\nu_n^1 \right)^{\frac{1}{1+\eps}}\cdot
\left(\bar\nu^1_n(M)\right)^{\frac{\eps}{1+\eps}} \nonumber\\
&\leq \left(\int_{M} |f\circ T^n|^{1+\eps} d\nu \right)^{\frac{1}{1+\eps}}\cdot
\left(\bar\nu^1_n(M)\right)^{\frac{\eps}{1+\eps}} \nonumber\\
&\le C_{g,\eps}(f)\cdot
\left(\bar\mu_n(M)\right)^{\eps/({1+\eps})}
\leq C_{g,\eps}(f)\cdot C_{\bc} \Lambda^{-s n\eps/(1+\eps)}.
\end{align}\vspace{.1cm}

Combining the three estimations in the above, we get that for all $n\geq 1$,
\begin{align*}
&\left|\int_M f\circ T^{n} d\nu-\int_M f\, d\mu\right|
=\left|\int_M f\circ T^n d\nu^1-\int_M f\circ T^n\, d\nu^2\right|\\
\leq&\sum_{m=1}^{n} | \left(\int_{ \cW^1_m} f\circ T^n d\nu^1_m-\int_{
\cW^2_m} f\circ T^n\, d\nu^2_m\right)|
+\left|\int_{ \bar{\cW}^1_n} f\circ T^n d\bar{\nu}^1_n\right|+\left|\int_{
\bar{\cW}^2_n} f\circ T^n\, d\bar{\nu}^2_n\right|\\
\leq&2C_{\bc}\|f\|_{C^\gamma}\cdot \Big(\sum_{m=1}^{n} \Lambda^{-\gamma(n-m) -sm}\Big)
+\|f\|_{p}\cdot C_{\bc}\Lambda^{-s n/q}
+C_{g,\eps}(f)\cdot C_{\bc} \Lambda^{-s n\eps/(1+\eps)}\\
\leq&2C_{\bc}\|f\|_{C^\gamma}
\Big(\frac{\Lambda^{-n\gamma/2}}{1-\Lambda^{-\gamma}}+\frac{\Lambda^{-ns/2}}{1-\Lambda^{-\gamma}}\Big)
+\|f\|_{p}\cdot C_{\bc}\Lambda^{-s n/q}
+C_{g,\eps}(f)\cdot C_{\bc} \Lambda^{-s n\eps/(1+\eps)}.
\end{align*}
Letting $\xi=\min\big\{\frac{\gamma}{2}, \frac{s}{2}, \frac{s}{q}, \frac{\eps}{1+\eps}s\big\}$
and $C(\gamma, s) = \frac{1}{1-\Lambda^{-\gamma}}+\frac{1}{1-\Lambda^{-\gamma}}$, 
we get
that
\[
|\int_M f\circ T^{n+\bar n} d\nu-\int_M f\, d\mu|
\leq 2(C_{\bc}C(\gamma, s) \|f\|_{p,\gamma}+C_{g,\eps}(f)\Lambda^{-\xi n},
\]
for  any $n\ge 1$.
This completes the proof of the lemma.
\end{proof}
\vspace{.1cm}

\subsection{Proof of Theorem \ref{main2}}
Let $p,q \in (1,\infty]$,  with $\frac{1}{p}+\frac{1}{q}=1$, and $f,g\in \cH_{p}(\gamma)$.
We start with the case when $g$ is non-negative, and $\mu(g)>0$.
Consider the  standard family $\cE=(\cW^u,\mu)$.
Let $n\ge 1$.
In the following we make a partition of $n$ into  two subintervals of length $\lfloor n/2 \rfloor$ and $n- \lfloor n/2 \rfloor$.
Let $\bar g_n(\alpha)=\mathbb{E}(g\circ T^{-n}|W_{\alpha})$ be the conditional expectation of $g\circ T^{-n}$ on $W_{\alpha}$ with respect to $\mu$.
By the  H\"older continuity of $g$, we have
$$\sup_{\alpha\in \cA^u}\sup_{x\in W_{\alpha}} |g\circ T^{-n}(x) -\bar g_{n}(\alpha)|\leq c_M\|g\|_{C^{\gamma}}\Lambda^{-n\gamma}.$$
Here the constant $c_M>0$ is the upper bound of the length of all the unstable manifolds
that is introduced at the end of \S \ref{assumptions} (see Remark 1). This implies that
\begin{align*}
|  \Cov_n(f,g)| &=| \Cov_{n- \lfloor n/2 \rfloor}(f,g\circ T^{-\lfloor n/2 \rfloor},T)| \\
&=| \Cov_{n- \lfloor n/2 \rfloor}(f,\bar g_{\lfloor n/2 \rfloor}, T)+ \Cov_{n- \lfloor n/2 \rfloor}(f,(g\circ T^{-\lfloor n/2 \rfloor}-\bar g_{\lfloor n/2 \rfloor}), T)|\\
&\leq | \Cov_{n- \lfloor n/2 \rfloor}(f,\bar g_{\lfloor n/2 \rfloor},T)| +c_M\mu(|f|)\cdot\|g\|_{C^{\gamma}}\Lambda^{-\lfloor n/2 \rfloor\gamma}.
\end{align*}

Next we consider $ \Cov_{n- \lfloor n/2 \rfloor}(f,\bar g_{\lfloor n/2 \rfloor},T)$. For any $r\in [\bs_0/2,\bs_0])$,
Let $\chi=\frac{\gamma_0\ln \Lambda}{r}$, 
$\cA_n'=\{\alpha\in \cA^u\,:\, |W_{\alpha}|\geq e^{-n \chi/2}\}$,
and $\cW_n'=\{T^{\lfloor n/2 \rfloor}W_{\alpha}: \alpha\in \cA_n' \}$.
Then the compliment of $\cA_n'$ satisfies
\beq\label{smallalpha}
\lambda^u((\cA_n')^c)=\mu(|W_{\alpha}|< e^{-n\chi}, \alpha\in \cA^u)\leq C_{\bp} e^{-r n \chi }
=C_{\bp}\Lambda^{- n \gamma_0},
\eeq
since $\cE$ is $r$-proper.

For any $\alpha\in \cA_n'$, 
we have $\cZ_{r}(W_{\alpha},\mu_{\alpha})\leq  e^{r n\chi/2}$ since $|W_{\alpha}|\geq  e^{-n\chi/2}$. 
Then Lemma \ref{properagain} implies that $T^{\lfloor n/2 \rfloor}(W_{\alpha}, \mu_{\alpha})$ is an $r$-proper family for any $\alpha\in \cA_n'$.
Let $\nu_n'=\int_{\cA_n'}\bar g_{\lfloor n/2 \rfloor} T^{\lfloor n/2 \rfloor}_*\mu_{\alpha} d\lambda^u(\alpha)$.
Note that
the function $\bar g_{\lfloor n/2 \rfloor}$ is constant on $T^{\lfloor n/2 \rfloor}W_{\alpha}$, $\alpha\in \cA_n'$.
It follows that $(T^{\lfloor n/2 \rfloor}\cW_n', \nu_n')$ is $r$-proper.
Applying Lemma \ref{decayinf} to the $r$-proper family $(T^{\lfloor n/2 \rfloor}\cW_n', \nu_n')$, we have
\begin{align*}
 \left|\int_{\cW_n'} f\circ T^{n-\lfloor n/2 \rfloor} \,d\nu_n'  -\mu(f) \right|
&\le  \left|\int_{\cW_n'} f\circ T^{n-\lfloor n/2 \rfloor} \,d\nu_n'  -\mu(f)\cdot \nu_n'(M) \right| 
+ C\mu(f)\cdot \mu(g)\cdot\Lambda^{- n \gamma_0} \\
&\leq [2C_{\bc}C(\gamma, s)\mu(g)\|f\|_{p,\gamma} +C_{g,\eps} (f)+ C\mu(f)\mu(g)]\Lambda^{-\frac{\xi n}{2}},
\end{align*}
for all $n\ge 1$.

We define $d\lambda_g(\alpha):=\bar g_{\lfloor n/2 \rfloor} d\lambda^u(\alpha)$, for any $\alpha\in (\cA_n')^c$.
 Thus by (\ref{bound}), we get for $n\geq 1$,

\begin{align}\label{Ng12}
&|\int_{\alpha\in (\cA_n')^c}  \int_{W_{\alpha}} \bar g_{\lfloor n/2 \rfloor}(\alpha)  f\circ T^{n-\lfloor n/2 \rfloor} \,d\mu_{\alpha} \,d\lambda^u(\alpha)|\nonumber\\
&=|\int_{\alpha\in \cA}  \int_{W_{\alpha}}   f\circ T^{n-\lfloor n/2 \rfloor} \cdot \chi_{(\cA_n')^c}(\alpha)\,d\mu_{\alpha}\,d\lambda_g(\alpha)|\nonumber\\
&\leq \left(\int_M  |f|^{1+\eps}\circ T^{n } \cdot  g(x)  \, d\mu\right)^{\frac{1}{1+\eps}} \lambda_g((\cA_n')^c)^{\frac{\eps}{1+\eps}}\nonumber\\
&\leq C_{g,\eps} (f) \Lambda^{-\frac{q n\eps \gamma_0}{1+\eps}},
\end{align}
where we used $\chi=\frac{\gamma_0\ln \Lambda}{r}$ in the last step.

Combining the above estimations, we get for $n\geq 1$,
\begin{align*}
|  \Cov_n(f,g)| 
&\leq C_{\eps,\kappa} K_f \Lambda^{-\frac{q n\eps \gamma_0}{1+\eps }}
+c_M\mu(|f|)\cdot\|g\|_{C^{\gamma}}\Lambda^{-\lfloor n/2 \rfloor\gamma} \\
&+ [2C_{\bc}C(\gamma, s)\mu(g)\|f\|_{p,\gamma} + C_{g,\eps} (f)+ C\mu(f)\mu(g)]\Lambda^{-\frac{\xi n}{2}}\\
&\leq [ 2C_{g,\eps} (f)+(2C(\gamma, s)C_{\bc}+c_M+C_{\bp})  (\mu(|f|)+\|f\|_{p,\gamma} )(\mu(g)+\|g\|_{C^{\gamma}} ) ] \vartheta^n,
\end{align*}
where $\vartheta=\max\{\Lambda^{-\frac{q \eps \gamma_0}{1+\eps }}, \Lambda^{-\frac{\xi }{2}}, 
\Lambda^{-\gamma/2}\} $.

In the case that $g$ is not non-negative,
we let $g=g^+ -g^- $ be the decomposition of $g$ into its positive and negative
parts. Then the above estimation holds for both $g^+ $ and $g^- $. Therefore we have for $n\geq N_g$,
\begin{align*}
&|\int_{M}f\circ T^n\cdot g\,d\mu-\mu(f)\cdot\mu(g)|
=|\int_{M}f\circ T^n\cdot
(g^{+}-g^{-})\,d\mu-\mu(f)\cdot\mu(g^{+}-g^{-})|\\
\le&|\int_{M}f\circ T^n\cdot g^{+}\,d\mu-\mu(f)\cdot\mu(g )|+
|\int_{M}f\circ T^n\cdot g^{-}\,d\mu-\mu(f)\cdot\mu(g )|\\
\leq&  [ 2C_{g,\eps} (f) +4(2C(\gamma, s)C_{\bc}+c_M+C_{\bp})\|f\|_{p,\gamma}\|g\|_{p,\gamma} ]\vartheta^n.
\end{align*}

For $f,g\in \cH_{\infty}(\gamma)$, one can check that (\ref{Ng12}) holds  for $n\geq 1$. Thus we have
$$|\int_{M}f\circ T^n\cdot g\,d\mu-\mu(f)\cdot\mu(g)|\leq C(\gamma) \|f\|_{\infty,\gamma}\|g\|_{\infty,\gamma}\vartheta^n,$$ for some constant $C(\gamma)>0$, and $n\geq 1$.
This completes the proof.

\section{Proof of Theorem \ref{main1}}

In this section our dynamical system $(M, T, \mu)$ satisfies the assumptions 
(\textbf{H1})--(\textbf{H6}).
\subsection{The standard family associated to a density function in $\cH_{\kappa,p}(\gamma)$.}
In this subsection, we assume $p\in(1,\infty]$ and $\gamma\in[\bgamma_0,1]$, $g\in \cH_{\kappa,p}(\gamma)$ is
a nonnegative, dynamically H\"older function with $0<\mu(g)<\infty$.
We have shown that  $g$  induces a standard family $\cG_{g}=(\cW,\nu)$, with $d\nu=gd\mu$. Moreover $g\circ T^{-1}$ generates $T\cG_g=(T\cW, T_*\nu)$.
Now we will investigate the $\cZ$-function of $T\cG_{g}$, and show that ${g\circ T^{-1}}$
leads to an $r$-standard family with $r\leq \bs_0$.
The main reason we consider $T\cG_{g}$ instead of $\cG_g$ is that the former can be $r$-proper even the latter is not, mainly because of the one-step expansion property.
We start with  a technical lemma that is useful later when defining the constant $C_{\bp}$ in \eqref{Cp2}
and when estimating $\cZ$-functions in Eq.~\ref{normest}. 
Recall the constants $b$ and $\bs$ given  in {\bf (H6)}.

\begin{lemma}\label{convergerp}
The series $A(r):=\sum_{n\geq 1} n^{-(\bs+2)+r(2-b)+\kappa/p}$ is convergent for every $r\in (0, \bs_0]$
when $b\ge 2$, and for every $r\in (0, \bs_0)$ when $b\in [1, 2)$. 
\end{lemma}
\begin{proof}
Let $h(r)=(\bs+2)-r(2-b)-\kappa/p$.
We divide the proof into two cases:
\begin{enumerate}  
\item[(1).] when $b \ge 2$: then for any $r\in (0, \bs_0]$,
$$h(r) \ge h(0) = \bs+2-\kappa/p \ge \bs +1 >1,$$ 
since $\kappa/p\leq 1$.  Therefore, the series  $A(r):=\sum_{n\geq 1} n^{-h(r)}$ converges for $r\in (0, \bs_0]$.

\item[(2).] when $b\in [1,2)$: then for any $r\in (0, \bs_0)$,
Then we have 
$$h(r) > h(\bs_0)=1+\bs-\bs_0(2-b) \ge 1,$$
\end{enumerate}
since $\bs \ge \bs_0(2-b)$ by (\ref{assumpb}). 
Therefore, the series $A(r):=\sum_{n\geq 1} n^{-h(r)}$ converges for $r\in (0, \bs_0)$.

Collecting terms, we complete the proof.
\end{proof}

Next we  reintroduce our definition of $r$-proper families for systems satisfying the new
hypothesis (\textbf{H6}). 
Let   $c>0$, $C_z>0$ be the
constants given by Lemma \ref{growth}.  
 We now define\footnote{The reason for this choice will be clear when estimating $\cZ_{\bs_0}(\cE)$ in 
 \eqref{Zs0E}. }
\beq\label{Cp2}
C_{\bp}(r)=\max\Big\{\cfrac{C_z}{1-c\Lambda^{-\gamma_0}}+ \frac{ C_A}{\bs+1-\bs_0(2-b)},\,\,\,
C_A(1+A(r))\Big\}
\eeq
for each $r\in (0, \bs_0]$.  Note that the term $\cZ_r(\cE)$ is left out intentionally from \eqref{Cp2}.

A potential problem is that
{it is possible to have } $\lim_{r\to\bs_0} C_{\bp}(r)=\infty$. To have a uniform control of standard families, we need to fix an $r_0\in (0,\bs_0]$ by
\begin{align}\label{defnr0}
r_0= \begin{cases}\bs_0, & \text{  when  } b\geq 2; \\ \frac{1}{2(b+\bs)}+\frac{\bs_0}{2},  &\text{ when } b\in [1,2).
\end{cases}
\end{align}
It follows from  (\ref{assumpb}) that $r_0 < \bs_0$ when $b \in [1, 2)$. 
Therefore, $C_{\bp}:=C_{\bp}(r_0) <\infty$ by Lemma~\ref{convergerp}.
An $r$-standard family $\cG$  is called \emph{$r$-proper} if $\cZ_r(\cG)<C_{\bp}$.

It follows from the definition that $\cZ_r(\cG)\leq \cZ_{r_0}(\cG)$  for any  $0< r \leq r_0$. 
Therefore, $\cG$ being $r_0$-proper implies that it is $r$-proper  for all  $0< r \leq r_0$.

\begin{lemma}\label{rproperg}
Let $p>1$, and $g\in \cH _{\kappa,p}(\gamma)$ be a probability density function. 
Define $\bar g(x)=\mathbb{E}(g|D_n)$  for any $x\in D_n$ and $n\ge 1$.
Then for $r\in (0, \bs_0)$:
\begin{enumerate}
\item[(1)] $T\cG_{\bar g}$ is $r$-standard and  
$\cZ_r(T\cG_{\bar g})\leq  C_{\bp}(r) K_g$.

\item[(2)] $T\cG_{g}$ is $r$-standard and 
$\cZ_r(T\cG_{g})\leq  C_{\bp}(r)(K_g+\|g\|_{\gamma})$.

\item[(3)] $T^{n+1}\cG_{g}$ is  $r_0$-proper
for $n\geq N_{g}:= \lceil \frac{1}{\gamma_1}\log_{\Lambda}^{+}(K_g+\|g\|_{\gamma}) \rceil$. 
In particular, $T\cG_{g}$ is $r_0$-proper if  $K_g+\|g\|_{\gamma}\leq 1$.

\item[(4)]  $\cE$ is $\bs_0$-proper.

\end{enumerate}
\end{lemma}
\begin{proof}

(1). Note that $K_g< \infty$ since $g\in \cH_{\kappa,p}$. It follows the definition of $\bar g$  that
\beq\label{Dnbargvalue}
|\bar g(T^{-1}x)|\leq K_g n^{\kappa/p},
\eeq
for any $x\in TD_n$, $n\geq 1$. 
According to (\ref{part}),  $D_n$  has an unstable foliation  such that each leaf $W_{\alpha}$ in this foliation satisfies:
$$\mu(D_n)/|TW_{\alpha}|\leq C_A n^{-\bs-b}.$$
Then by Lemma \ref{convergerp}, we have
\begin{align}
\cZ_r(T\cG_{\bar g})
&=\sum_{n\geq 1}\int_{\alpha\in \cA_n}\frac{   |\bar g(T^{-1}x_n)|}{|TW_{\alpha}|^{r}}\,\cdot d\lambda(\alpha)
\leq C_AK_g\sum_{n\geq 1} n^{\kappa/p}|TW_{\alpha}|^{-r} \mu(D_n)\nonumber\\
&\leq C_AK_g\sum_{n\geq 1} n^{\kappa/p}|TW_{\alpha}|^{-r} \mu(D_n)^r \mu(D_n)^{1-r}
\leq C_AK_g\sum_{n\geq 1} n^{-(\bs+2)(1-r)-(\bs+b)r+\kappa/p}\nonumber\\
&=C_AK_g\sum_{n\geq 1} n^{-(\bs+2)+r(2-b)+\kappa/p} \leq C_{\bp}(r)K_g , \label{normest}
\end{align}
where $x_n\in TW_{\alpha}$. 
This completes the proof of (1).

(2). For any $x\in TD_n$,
\beq\label{bargg}
|g(T^{-1}x)-\bar g(T^{-1}x)|\leq \|g\|_{\gamma}\cdot  \diam(D_n)^{\gamma}.
\eeq
Then (\ref{nugalpha}) implies that
 \begin{align}\label{cZgbarg}
&|\cZ_{r}(T\cG_{g})-\cZ_{r}(T\cG_{\bar g})|
\leq \sum_{n\ge 1} \int_{\alpha\in\cA_n}\frac{|g\circ T^{-1}-\bar g\circ T^{-1}|}{|W_{\alpha}|^{r}}\,d\lambda^u(\alpha)\nonumber\\
\leq &\|g\|_{\gamma}\sum_{n\ge 1}  \int_{\cA_n}\frac{\diam(D_n)^{\gamma}}{|TW_{\alpha}|^{r}}\,d\lambda(\alpha)
\leq \|g\|_{\gamma}\cZ_r(\cE).
 \end{align}
Combining with item (1), we get
$$\cZ_r(T\cG_{{g}})\leq C_{\bp}(r)(K_g +\|g\|_{\gamma}).$$

In the case $g\in \cH_{\infty} (\gamma)$, we have
$\cZ_{r}(\cG_g)\leq \|g\|_{\infty} \cZ_{r}(\cE)<\infty$. 

(3). The statement directly follows from (\ref{firstgrowth}). 

(4) In the estimation (\ref{normest}), we take $g\equiv 1$.  Then  $\kappa=0$, $K_g=1$ and
\begin{align}\label{normest1}
\cZ_r(\cE)&=\cZ_r(T\cE)=\sum_{n\geq 1}\int_{\alpha\in \cA_n}\frac{ 1}{|TW_{\alpha}|^{r}}\,\cdot d\lambda(\alpha)\nonumber\\
&\leq C_A\sum_{n\geq 1} n^{-(\bs+2)+r(2-b)}.
\end{align} 
Note that for $b\geq 2$, we have 
$\bs+1-\bs_0(2-b)>0$. On the other hand, for $b\in [1,2)$, (\textbf{H6}) implies that $s+b>1/\bs_0$. Thus $\bs>1-b=2-b-1>\bs_0(2-b)-1$. This implies that
\begin{align}\label{Zs0E}
\cZ_{\bs_0}(\cE)\leq C_A\sum_{n\geq 1} n^{-(\bs+2)+\bs_0(2-b)}
\leq \frac{ C_A}{\bs+1-\bs_0(2-b)}\leq C_{\bp}.
\end{align}
Combining above facts, we have shown that $\cE$ is $\bs_0$-proper. 
\end{proof}

Note that above lemma implies that for any $p>1$ and for any dynamically H\"older function
$g\in \cH_p(\gamma)$ with $\mu(g)<\infty$, the associated standard family
$T\cG_{g}$ is $r_0$-standard.

\vskip.3cm

To prove Theorem \ref{main1}, it is enough to verify condition  (\ref{bound}) under assumption (\textbf{H6}). 
\begin{proposition}\label{coro-main3}
Let $f,g\in \cH_{\kappa,p}(\gamma)$, and   $g$ induces an $r$-standard family family $\cG=(\cW,\nu)$, 
with $r\in [r_0, \bs_0]$. 
Then the following holds:
 \beq\label{regularcase}|
 T^n_*\nu(|f|^{1+\eps_p})|\leq C_{g,\eps_p}(f):=2\eps_p^{-1}C_{\bp} C_A^{r_0} K_f^{1+\eps_p}(K_g+\|g\|_{\gamma}),
\eeq
for any $n\geq 1$, with   $\eps_p=((\bs+b) \bs_0-1)/4$.
\end{proposition}
\begin{proof}
 Let $\cG=(\cW, \nu)$ be the $r$-standard
family generated by $g\in \cH_{\kappa,p}(\gamma)$, according to  Lemma \ref{rproperg}. 
Let $\eps>0$. Note that for $k\ge 1$,
$k^{1+\eps} \le \sum_{l=1}^k 2 l^{\eps}$ (by induction on $k$). Since $f\in  \cH_{\kappa,p}(\gamma)$, $| f(x)|\leq K_f k^{\kappa/p}$  for any $x\in D_k$ and $k\geq 1$. Then we have
\begin{align}
T^n_* \nu(| f|^{1+\eps})&\leq K_f^{1+\eps}\sum_{k\ge 1} k^{(1+\eps)\kappa/p}\cdot T^n_* \nu(D_k)
\leq 2K_f^{1+\eps}\sum_{k\ge 1} \sum_{l=1}^k l^{(1+\eps)\kappa/p-1}\cdot T^n_* \nu(D_k) \nonumber \\
&\leq  2K_f^{1+\eps}\sum_{l\ge 1} l^{(1+\eps)\kappa/p-1}\cdot T^n_* \nu(\cup_{k\geq l}D_k). \label{reorder}
\end{align}
Now we need to estimate the measures $T^n_*\nu(\cup_{k\geq l}D_k)$ for each $l\ge 1$.

By  (\ref{part}), there exist $C_A>0$ and $b>0$  such that
$ r^u(x)\leq C_A  \cdot  m^{-b-\bs}$ for all $x\in TD_m$. Lemma \ref{rproperg}
 and Eq.~(\ref{standard}) imply that
\begin{align}\label{Tnnu}T^n_*\nu(\cup_{m\geq l}D_k)&\leq \sum_{m\geq l} T^n_*\nu( D_m)=\sum_{m\geq l} T^{n-1}_*(T_*\nu)( TD_m))\nonumber\\
&\leq T^{n-1}_*(T_*\nu)(r^u<C_A  \cdot  l^{-(b+\bs)})\leq C_{\bp}(r) C_A^r(K_g+\|g\|_{\gamma})\cdot  l^{-(b+\bs) r}.
\end{align}
Combining above facts, we get
\begin{align}\label{Tnnurp}
T^n_* \nu(| f|^{1+\eps})&\leq  2K_f^{1+\eps}C_{\bp}(r) C_A^r\sum_{l\ge 1} l^{(1+\eps)\kappa/p-(b+\bs) r-1}\nonumber\\
&\leq  2K_f^{1+\eps}C_{\bp}(r) C_A^r(K_g+\|g\|_{\gamma})\sum_{l\ge 1} l^{-1-((b+\bs) r-1-\eps)}
\end{align}

It follows from (\ref{assumpb}) that $(\bs+b)\bs_0-1>0$. Let 
\beq\label{defnepsp}
 \eps_p:=\frac{1}{4}((\bs+b) \bs_0-1)>0.
\eeq 
It is easy to see $(b+\bs) \bs_0-1-\eps_p=3\eps_p$. 
When $b\in[1,2)$, $(b+\bs) r_0-1-\eps_p=\eps_p$.
where  $ r_0=\frac{1}{2(b+\bs)}+\frac{\bs_0}{2}$ is defined in \eqref{defnr0}.
  This implies that $(b+\bs) r-1-\eps_p \ge \eps_p$ for any $r\in [r_0, \bs_0]$.

Combining with  \eqref{Tnnurp} and setting $r=r_0$,  we have 
\begin{align}\label{Cffpre}
T^n_* \nu(|f|^{1+\eps_p})&\leq 2K_f^{1+\eps_p}C_{\bp}(r_0) C_A^{r_0}(K_g+\|g\|_{\gamma})\sum_{l\ge 1} l^{-1-((b+\bs) r_0-1-\eps_0)}\nonumber\\
&=  2K_f^{1+\eps_p}C_{\bp}(r_0) C_A^{r_0}\sum_{l\geq 1} l^{-1-\eps_p}\leq 2\eps_p^{-1}K_f^{1+\eps_p}C_{\bp}(r_0) C_A^{r_0}(K_g+\|g\|_{\gamma}),
\end{align}
where the last constant is exactly $C_{g,\eps_p}(f)$.  This completes the proof.

\end{proof}

Now Theorem \ref{main2} follows from the proof of Theorem \ref{main1}.

\section{Return time function for the induced maps of billiards}\label{infinitevariance}

The specific billiard systems considered here include semi-dispersing
billiards,  billiards with cusps, and Bunimovich Stadia, whose basic properties have been studied in
\cite{Ma04,CZ05a,CZ05b,CZ07,CZ08,CM07}. Let $Q$ be a
billiard table and $\cM=\partial Q\times[-\pi/2,\pi/2]$ be the phase space of
the billiard map $\cF$ induced  on $Q$, which preserves a smooth measure
$\mu_{\cM}=|\partial Q|^{-1} \cdot \cos\varphi\, dr\, d\varphi$. The reduced phase
space $M$ consists of all/some collisions on dispersing parts of $\partial Q$ and
all/some first collisions on the focusing parts of $\partial Q$. The conditional measure obtained by
restriction of $\mu_{\cM}$ on $M$ is denoted by $\mu$, which is a smooth invariant
measure (certainly an SRB measure). Let $T$ be the first return map of $\cF$ from $M$
to $M$. Then it was proved in the above references that the induced system
$(T,M,\mu)$ satisfies the assumptions \textbf{(H1)--(H5)}. In fact, the
one-step expansion estimate \textbf{(H5)} actually holds for $\bs_0=1$. Note that
the discontinuities of $R$ can only occur at the singularities of $T$. We define $\{D_m\}$ as in (\textbf{H6}), to be an enumeration of level sets of $R$. 

Since for large $m$,  the set $D_m$ is bounded by two stable curves and two unstable curves. One can foliate $D_n$ into relatively long unstable curves $W_{\alpha}$, with length satisfying 
$$c_A m^{-b-\bs}\leq |W_{\alpha}|\leq C_A m^{-b-\bs}$$
Where $b$ and $\bs$ will be determined for each billiard system below. 
By time-reversibility of billiard system, the set $TD_m$ is symmetric with respect to $\varphi$ coordinate  about  $D_m$.  Thus the stable dimension of $TD_m$ is bounded above by $C_A m^{-b-\bs}$. Moreover, by the absolute continuity of the stable foliation, for any $\alpha\in \cA_n$, the length of $TW_{\alpha}$ is comparable to the length of the unstable dimension of $TD_m$. Also using the fact that $\mu$ has probability density bounded by $(|\partial Q|\mu_{\cM}(M))^{-1}$, we get 

\begin{align*}
\mu(D_m)/|TW_{\alpha}|\leq C m^{-b-\bs}
\end{align*}
This verifies the second inequality in (\textbf{H6}(ii)). Thus it is enough to verify (\ref{assumpb}) for each billiard system. 

\vspace{.2cm}

Note that the return time function $R$ is always in $L^1(\mu)$.
In fact, the Kac's formula gives that $\mu(R)=\mu_{\cM}(M)^{-1}$.
Moreover, $\mu(D_n)\asymp n^{-3}$ for the above billiard systems,
see \S \ref{veriBill} for more details.
The return time function $R$ satisfies $R\in L^{p}(\mu)$, for
any $1\le p<2$, while $R\notin L^{2}(\mu)$.
Therefore, the classical results
on the decay of correlations are not applicable in the study of the autocorrelations of the function $R$.
One needs some finer characterizations of the geometric feature of the components $D_n$, $n\ge 1$.

\subsection{Type A billiard systems.}\label{veriBill}

\noindent{\textbf{Case I. Semi-dispersing billiards.}} Billiards in a square
with a small fixed circular obstacle removed are known as semi-dispersing
billiards. Chernov and Zhang proved in \cite{CZ08} that this system has a decay rate
of correlations bounded by $\text{const} \cdot n^{-1}$. For this billiard
system, the reduced phase space $M$ is comprised of collisions with the
circular obstacle only. The induced map $T: M \rightarrow M$ is then equivalent to
the well studied Lorentz gas billiard map (cf. \cite{CZ05a}), which is known to
have exponential decay of correlations (see \cite{CM07}, for instance). The properties \textbf{(H1)--(H6)}
 were verified in \cite{BSC90,BSC91,CM, CZ05a}. 
 according to \cite{CM}, the measure of each $m$-cell is $\mu(M_m) \asymp m^{-3}$, with $\bs=1$, $\bs_0=1$.
As a result, we see that
the return time map $R\notin L^2(\mu)$, while $R\in \cH_{\kappa,p}(1)$ for any $p<2$, and $\kappa=p$.
 Moreover, each $m$-cell $M_m$ has unstable-dimension $\leq C_A m^{-2}$; which implies $b=1$.  Thus the condition (\ref{assumpb} ) is also verified, as $b>1-\bs=0$, and $1=\bs\geq \bs_0(2-b)=1$. Therefore, the standard family $(\cE, \nu)$ with $d\nu = R\circ T^{-1} d\mu$ is $r_0$-proper with $r_0=\frac{1}{2(b+\bs)}+\frac{\bs_0}{2}=3/4$.\\

 \noindent{\textbf{Case II. Billiards with cusps.}}
The billiards with cusps were first studied by Machta (cf. \cite{Ma83}). It is known
that the billiard maps on these tables are hyperbolic and ergodic. However,
the hyperbolicity is non-uniform due to the collisions {\it deep down the cusps},
where the particle experiences a large number of rapid collisions in a short amount of  time.
It is proved in \cite{CZ05a,CZ08, CM07} that the system enjoys correlations decay of order $n^{-1}$. Moreover, it was proved that that all properties of \textbf{(H1)--(H6)}   hold for the induced map $(T,M,\mu)$.
\vspace{.2cm}

For this billiard system, the reduced subspace $M\subset \cM$ consists of
collisions whose free paths are {\it not that short}. Let $R$ be the first
return time function and $D_m$ be the cell induced by $R$ for each $m\ge 1$. Then
each $m$-cell $TD_m$ has measure $\mu(D_m)\asymp m^{-3}$,
stable-diameter $\asymp m^{-7/3}$ and unstable-diameter $\asymp m^{-2/3}$.
So we can set $\bs=1$, $b=2$. Thus the condition (\ref{assumpb} ) is also verified, 
as $b>1-\bs=0$, and $1=\bs\geq \bs_0(2-b)=0$. Therefore, the standard family $(\cE, \nu)$ with $d\nu = R\circ T^{-1} d\mu$ is $r_0$-proper with  $r_0=\bs_0=1$.\\

 \noindent{\textbf{Case III. Dispersing billiards with flat points.}}
A family of dispersing billiards  with finite number of flat points (of zero
curvature) were first studied in \cite{CZ05b}, where they proved that the
correlations decay at rate of $\mathcal{O}(n^{-\alpha})$, where
$\alpha=\frac{\beta+2}{\beta-2}$ depends on the {\it flatness} parameter
$\beta>2$ of the billiard tables. Moreover, it was proved that that all properties of
\textbf{(H1)--(H6)}   hold for the induced map $(T,M,\mu)$.
\vspace{.2cm}

Let $R$ be the first return time function of the billiard map with respect
to $M$ and $D_m$, $m\ge 1$,
be the level sets of $R$. The measure of each $m$-cell satisfies $\mu(D_m) \asymp
m^{-3-\frac{4}{\beta-2}}$.  In this case, the return
time function $R\in \cH_{\kappa,p}(1)$ for any $p<2+\frac{4}{\beta-2}$, and $\kappa=p$.
Moreover,  $TD_m$ has unstable-dimension of order $\cO(1)$, and stable-dimension of order $m^{-3-\frac{4}{\beta-2}}$, which implies that $\bs=\frac{\beta+2}{\beta-2}$ and $b=\bs+1>2$.Thus the condition (\ref{assumpb} ) is also verified, 
as $b>1-\bs=0$, and $1=\bs>0>\bs_0(2-b)$. Therefore, the standard family $(\cE, \nu)$ with $d\nu = R\circ T^{-1} d\mu$ is $r_0$-proper with  $r_0=\bs_0=1$.

\subsection{Type B billiards}
\label{pro.typB}

Now we consider the type B billiard systems. For this type of billiards, we assume:
\begin{itemize}
\item[(h1)] Each cell $D_m$ has measure
$\mu(D_m)\asymp m^{-3}$. 
\item[(h2)] The set $TD_m$ has stable-diameter $\asymp m^{-2}$ and unstable-diameter $\asymp m^{-1}$.
\end{itemize}

The Bunimovich stadium and the skewed stadium are two typical examples of this type of systems. The stadium billiard table, introduced by Bunimovich in \cite{Bu74}, is
comprised of two equal semicircles which are connected by two parallel lines. The boundary of the skewed stadium consists two  equal circular arcs (with different radius), which are connected by two tangential lines. 
It has been shown that both stadia are nonuniformly hyperbolic,
ergodic, and mixing (cf. \cite{Bu74, Bu79, CM07} for discussions and results).
Let $M$ be the set of first collisions on the  arcs of the billiard
table, $R$ be the first return time function on $M$, and $T$ be the induced first return map. All assumptions in
\textbf{(H1)--(H5)} were checked in \cite{Bu79, CM, CM07, CZ08},. It is enough that we check the assumption (\ref{assumpb} ).
It was shown in \cite{Bu79, CM} that each cell $D_m$ has measure
$\mu(D_m)\asymp m^{-3}$. The set $TD_m$ has stable-diameter $\asymp m^{-2}$ and unstable-diameter $\asymp m^{-1}$. So the return time function $R\in \cH_{\kappa,p}(1)$, with $1\le p<2$, $\kappa=p$. Moreover, we have $\bs=\bs_1=\bs_0=1$, and $b=1$.  Thus the condition (\ref{assumpb} ) is also verified, as
 $b>1-\bs=0$, and $1=\bs\geq \bs_0(2-b)=1$. Therefore, the standard family $(\cE, \nu)$ with $d\nu = R\circ T^{-1} d\mu$ is $r_0$-proper with $r_0=\frac{1}{2(b+\bs)}+\frac{\bs_0}{2}=3/4$.

\section*{Acknowledgements}
{ We would like to thank the anonymous referees for careful  reading and constructive comments,
which helped substantially improving presentation of the manuscript.}
F. Wang is partially supported by the NSFC Grant (11101294),
the State Scholarship Fund from China Scholarship Council (CSC),
and "Youxiu Rencai Peiyang Zizhu" (Class A) from the Beijing City. H. Zhang is partially supported by a grant from the Simons Foundation.


\begin{thebibliography}{a1}

\bibitem{ABV}  A. Alves, C. Bonatti, and M. Viana.
\emph{SRB measures for partially hyperbolic systems whose central direction is
mostly expanding},
{Invent. Math.} \textbf{140} (2000), 351--398.

\bibitem{BCD}
P. Balint, N. Chernov, and D. Dolgopyat.
\emph{Limit theorems for dispersing billiards with cusps},
{Comm. Math. Phys}. \textbf{308} (2011), 479--510.


\bibitem{B1} R. Bowen.
\emph{Equilibrium States and the Ergodic Theory of Axiom A Diffeomorphisms},
{Lect. Notes in Math.} \textbf{470}, Springer-Verlag, Berlin-New York, 1975.

\bibitem{Bu74} L. A. Bunimovich.
\emph{On billiards close to dispersing},
{Math. USSR Sbornik}, \textbf{23} (1974), 45--67.

\bibitem{Bu79} L.A. Bunimovich.
\emph{On the ergodic properties of nowhere dispersing billiards},
{Commun. Math. Phys.} \textbf{65} (1979), 295--312.


\bibitem{BSC90}
L. A. Bunimovich, Ya. G. Sinai and N. I. Chernov.
\emph{Markov partitions for two-dimensional hyperbolic billiards},
{Russian Math. Surveys} \textbf{45} (1990), 105--152.

\bibitem{BSC91}
L. A. Bunimovich, Ya. G. Sinai and N. I. Chernov.
\emph{Statistical properties of two-dimensional hyperbolic billiards},
{Russian Math. Surveys}. \textbf{46} (1991), 47--106.



\bibitem{C99} N. I. Chernov.
\emph{Decay of correlations in dispersing billiards},
{J. Statist. Phys.} \textbf{94} (1999), 513--556.


\bibitem{C06}
N. I. Chernov.
\emph{Advanced statistical properties of dispersing billiards},
{J. Statist. Phys.} \textbf{122} (2006), 1061--1094.

\bibitem{CD} N. I. Chernov and D. Dolgopyat.
\emph{Brownian Brownian Motion-I},
{Memoirs of AMS.} \textbf{198} (2009).

\bibitem{CD09} N. I. Chernov and D. Dolgopyat.
\emph{Anomalous current in periodic Lorentz gases with infinite horizon},
Russian Math. Surveys, \textbf{64} (2009), 651--699.



\bibitem{CM}
 N. Chernov, and R. Markarian.
 \emph{Chaotic Billiards},
 {Math. Surveys Monographs} \textbf{127}, AMS, Providence, 2006.

\bibitem{CM07}
N. Chernov, and R. Markarian.
\emph{Dispersing billiards with cusps: slow decay of correlations},
{Comm. Math. Phys.} \textbf{270}, 2007, 727--758.



\bibitem{CZ05a}
N. Chernov and H.-K. Zhang.
\emph{Billiards with polynomial mixing rates},
{Nonlineartity} \textbf{4} (2005), 1527--1553.

\bibitem{CZ05b} N.~Chernov and H.-K.~Zhang.
\emph{A family of chaotic billiards with variable mixing rates},
{Stochast. Dynam.} \textbf{5} (2005), 535--553.


\bibitem{CZ07} N. Chernov and H.-K. Zhang.
\emph{Regularity of Bunimovich's stadia},
{Regular and Chaotic Dynamics} \textbf{3} (2007), 335--356.

\bibitem{CZ08}  N. Chernov and H.-K. Zhang.
\emph{Improved estimate for correlations in billiards},
{Comm. Math. Phys.}  \textbf{277}  (2008), 305--321.


\bibitem{CZ09}
N. Chernov and H.-K. Zhang.
\emph{On statistical properties of hyperbolic systems with singularities},
{J. Statist. Phys.} \textbf{136} (2009), 615--642.


\bibitem{DZ11} M. Demers and H.-K. Zhang.
\emph{Spectral analysis of the transfer operator for the Lorentz gas},
{ J. Modern Dynamics} \textbf{5} (2011), 665--709.

\bibitem{DZ13} M. Demers and H.-K. Zhang.
\emph{A functional analytic approach to perturbations of the Lorentz gas},
{Comm. Math. Phys.}  \textbf{324} (2013), 767--830.


\bibitem{DZ14} M. Demers and H.-K. Zhang.
\emph{Spectral analysis of hyperbolic systems with singularities},
{Nonlinearity} \textbf{27}   (2014), 379--433.

\bibitem{D01} D. Dolgopyat.
\emph{On dynamics of mostly contracting diffeomorphisms},
{Comm. Math. Phys.} \textbf{213} (2001), 181--201.



\bibitem{D04b} D. Dolgopyat.
\emph{On differentiability of SRB states for partially hyperbolic systems},
{Invent. Math.} \textbf{155} (2004), 389--449.


\bibitem{KS86} A. Katok and J.-M. Strelcyn.
\emph{Invariant Manifolds, Entropy and Billiards;
Smooth Maps with Singularities},
{Lect. Notes Math.} \textbf{1222}, Springer, New York 1986.


\bibitem{Ma83} J. Machta.
\emph{Power law decay of correlations in a billiard problem},
{J. Stat. Phys.}  \textbf{32} (1983), 555--564.

\bibitem{Ma04}
R. Markarian.
\emph{Billiards with polynomial decay of correlations},
{Ergod. Th. Dynam. Syst.} \textbf{24} (2004), 177--197.




\bibitem{P92} Ya. Pesin.
\emph{Dynamical systems with generalized hyperbolic attractors: hyperbolic,
ergodic and topological properties},
{ Ergod. Th. Dynam. Syst.} \textbf{12} (1992), 123--152.




\bibitem{SYZ} M. Stenlund,   L.S. Young, and H.-K. Zhang.
\emph{Dispersing billiards with moving scatterers},
{Comm. Math. Phys.} \textbf{322} (2013) 909--955.

\bibitem{Y98} L.S. Young.
\emph{Statistical properties of systems with some
hyperbolicity including certain billiards},
{ Ann. Math.} \textbf{147} (1998), 585--650.

\bibitem{Y99}  L.S. Young.
\emph{Recurrence times and rates of mixing},
{Israel J. Math.} \textbf{110} (1999), 153--188.

\bibitem{Y02} L.-S. Young.
\emph{What are SRB measures, and which  dynamical systems have them?}
{J. Statist. Phys.} \textbf{108} (2002), 733--751.

\bibitem{Zh11} H.-K. Zhang.
\emph{Current in periodic Lorentz gases with twists.}
{Commun. Math. Phys.} \textbf{306} (2011), 747--776.

 \end{thebibliography}
\end{document}